\documentclass[11pt,reqno]{amsart}
\usepackage[utf8]{inputenc}
\usepackage{amsmath,amssymb}
\usepackage{hyperref}
\usepackage{mathtools}
\usepackage{cite}
\usepackage{tikz}
\usetikzlibrary{automata}
\usepackage{amsmath}
\usepackage{mathrsfs}

\numberwithin{equation}{section}
\newtheorem{thm}{Theorem}[section]
\newtheorem{cor}[thm]{Corollary}
\newtheorem{lem}[thm]{Lemma}
\newtheorem{prop}[thm]{Proposition}
{\theoremstyle{remark}
\newtheorem{rmk}[thm]{Remark}}
{\theoremstyle{example}
\newtheorem{example}[thm]{Example}}
\newcommand{\abs}[1]{\left\vert#1\right\vert}
{\theoremstyle{definition}

\newcommand{\malcev}{\mathbin{\hbox{$\bigcirc$\rlap{\kern-8.25pt\raise0,50pt\hbox{${\tt
  m}$}}}}}
\newcommand{\smalcev}{\mathbin{\hbox{$\bigcirc$\rlap{\kern-7pt\raise0,30pt\hbox{${\tt
  m}$}}}}}
\newcommand{\im}{\mathop{\mathrm{Im}}\nolimits}
\newcommand{\Aug}{\mathop{\mathrm{Aug}}\nolimits}
\newcommand{\J}{\mathrel{\mathscr J}} 
\newcommand{\R}{\mathrel{\mathscr R}} 
\newcommand{\eL}{\mathrel{\mathscr L}} 
\newcommand{\HH}{\mathrel{\mathscr H}}
\newcommand{\K}{\mathrel{\mathscr K}}
\newcommand{\rk}{\mathop{\mathrm{rk}}\nolimits}
\newcommand{\End}{\mathop{\mathrm{End}}\nolimits}
\newcommand{\rank}{\mathop{\mathrm{rank}}\nolimits}
\newcommand{\init}{\mathop{\mathrm{init}}\nolimits}
\newcommand{\ter}{\mathop{\mathrm{ter}}\nolimits}

\begin{document}
\title[Simplicity of augmentation submodules]{Simplicity of augmentation submodules for transformation monoids}
\author{M.H. Shahzamanian}
\author{B. Steinberg}
\address{M.H. Shahzamanian\\ Centro de Matem\'atica e Departamento de Matem\'atica, Faculdade de Ci\^{e}ncias,
Universidade do Porto, Rua do Campo Alegre, 687, 4169-007 Porto,
Portugal}
\email{m.h.shahzamanian@fc.up.pt}
\address{B. Steinberg  \\ Department of Mathematics\\
    City College of New York\\
    Convent Avenue at 138th Street\\
    New York, New York 10031\\
    USA.}
\email{bsteinberg@ccny.cuny.edu}

\begin{abstract}
For finite permutation groups, simplicity of the augmentation submodule is equivalent to $2$-transitivity over the field of complex numbers. We note that this is not the case for transformation monoids. We characterize the finite transformation monoids whose augmentation submodules are simple for a field $\mathbb{F}$ (assuming the answer is known for groups, which is the case for $\mathbb C$, $\mathbb R$, and $\mathbb Q$) and provide many interesting and natural examples such as endomorphism monoids of connected simplicial complexes, posets, and graphs (the latter with simplicial mappings).
\end{abstract}
\maketitle

\section{Introduction}
Over the past two decades there has been a resurgence of interest in the representation theory of finite monoids coming from a number of different sources. The main catalyst was a paper of Bidigare, Hanlon and Rockmore applying monoid representation theory to analyzing finite state Markov chains~\cite{BHR}.  This was followed by work of Brown and Diaconis~\cite{DiaconisBrown1} and then many others~\cite{Brown1,Brown2,bjorner2,GrahamChung,Saliolaeigen,sandpile,AyyerKleeSchilling,ayyer_schilling_steinberg_thiery.2013,randomwalksrings}.  These developments are discussed in the second author's recent book~\cite[Chapter~14]{Ben-Rep-Monoids-2016}. Connections with fast Fourier transforms and data analysis can be found in~\cite{Malandro1,Malandro2,Malandro3}.

There have also been a number of applications of the representation theory of monoids to studying finite dimensional algebras arising in discrete geometry from hyperplane arrangements, oriented matroids and CAT(0) cube complexes~\cite{DiaconisBrown1,Saliolahyperplane,MSS,ourmemoirs}.  Papers studying connections between the representation theory of finite monoids and the representation theory of finite dimensional algebras and quivers include~\cite{Putcharep3,DO,Saliola,rrbg,globaltn,itamar1}.  Applications to algebraic combinatorics and descent algebras can be found in~\cite{BidigareThesis,Brown2,SaliolaDescent,Hsiao,rrbg,MSS,ourmemoirs,Denton,Jtrivialpaper,BergeronSaliola}.

Representation theory of finite monoids can be used as a tool to study finite monoids acting on finite sets, cf.~\cite{Ben-Transformation-Monoids-2010} and~\cite[Chapter~13]{Ben-Rep-Monoids-2016}.  A transformation monoid has an associated transformation module (analogous to the way a permutation group has an associated permutation module).  When analyzing Markov chains using monoid representation theory, it is precisely transformation modules that are used.
Another source of applications of transformation modules is to automata theory~\cite{EilenbergA}.  Associated to any finite state automaton is a transformation monoid.  Representation theoretic aspects of the corresponding transformation module can often be exploited to study the automaton, cf.~\cite{Perrincomplred,berstelperrinreutenauer}.  This has particularly, been the case for study of the \v{C}ern\'y conjecture~\cite{cerny}, an over 60 year old problem in automata theory.  In mathematical terms, it asserts that if $A$ is a set of mappings on an $n$-element set such that the monoid generated by $A$ contains a constant mapping, then there is a product of at most $(n-1)^2$ elements of $A$ (with repetitions allowed) that is a constant map; see~\cite{VolkovLata} for a nice survey.  Many of the deepest results concerning the \v{C}ern\'y conjecture exploit the transformation module and its augmentation submodule, that is, the submodule consisting of those vectors whose coordinates sum to zero.  See, for instance,~\cite{Kari,Pincerny,rystsov1,dubuc,primecycle,averagingjournal,bealperrinnew,synchgroups}.

More precisely, if $(M,\Omega)$ is a finite transformation monoid and $\mathbb{F}$ a field, then $\mathbb F\Omega$ is the transformation module, where the action of $M$ on $\Omega$ is extended linearly.  The augmentation submodule $\Aug(\mathbb F\Omega)$ consists of those formal linear combinations of elements of $\Omega$ whose coefficients sum to zero.
It is a classical result going back to Burnside that if $G$ is a transitive permutation group on $\Omega$, then $\Aug(\mathbb C\Omega)$ is simple if and only if $G$ is $2$-transitive.  The permutation $(G,\Omega)$ is $2$-homogeneous, that is, acts transitively on unordered pairs of elements of $\Omega$, if and only if $\Aug(\mathbb R \Omega)$ is simple, cf.~\cite{Ar-Ca-St,cameron}.  Motivated by the results of~\cite{synchgroups}, the second author asked John Dixon about when $\Aug(\mathbb Q\Omega)$ is simple.  Dixon made partial progress on this question in~\cite{Dix}, where he showed that such permutation groups are primitive of either affine type or almost simple and classified the examples of affine type.  He also showed, that if $\Aug(\mathbb Q\Omega)$ is simple, then $G$ is $3/2$-transitive.  The classification of $3/2$-transitive groups was obtained in~\cite{Bam} and, in particular, the classification of permutation groups with $\Aug(\mathbb Q\Omega)$ simple was completed in~\cite[Corollary~1.6]{Bam}.

In a number of the applications of transformation modules to the study of transformation monoids, in particular to synchronization and bounding the lengths of synchronizing words~\cite{synchgroups,Ar-Ca-St,Ben-Transformation-Monoids-2010}, the simplicity of the augmentation module plays an important role. It has been a challenging question to determine when the augmentation is simple and to clarify  what is the relationship between simplicity of the augmentation submodule for finite transformation monoids that are not groups and $2$-transitivity.

 In this paper, we characterize when the augmentation module of a transformation monoid is simple over a field, assuming that the answer is known for permutation groups, as is the case for the fields $\mathbb C$, $\mathbb R$ and $\mathbb Q$.  We show that it is not the case that all $2$-transitive transformation monoids have simple augmentation modules over $\mathbb C$;  it turns out there is an extra condition that the incidence matrix of a certain set system should have full rank.  We also show that $2$-transitivity is not necessary either.  In the process we show that a plethora of naturally arising transformation monoids  in combinatorics have simple augmentation modules over any field. Examples include endomorphism monoids of connected simplicial complexes and connected posets (or equivalently, monoids of continuous self-maps of finite, connected $T_0$ topological spaces). In particular, if $\Gamma$ is a connected but not complete graph, then the monoid of all simplicial endomorphisms of $\Gamma$ has a simple augmentation module over any field but is not $2$-transitive.
Our techniques involve a mixture of monoid representation theory and combinatorics.

The paper is organized as follows.  We being by recalling background on monoids and their representation theory, as well as on combinatorial structures like finite simplicial complexes, directed graphs and posets, so as to make the paper accessible to as broad an audience as possible.  We then present in the following section our characterization of transformation monoids with a simple augmentation module.  This is followed by a section of examples of transitive transformation monoids with and without simple augmentation modules.  A key role is played by edge transitive monoids of simplicial endomorphisms of a connected graph and by set systems and partial orders.  The final section considers the case of $0$-transitive transformation monoids, or equivalently, transitive partial transformation monoids. Here key examples included endomorphism monoids of meet semilattices and certain Rees matrix semigroups.


\section{Preliminaries}
\subsection{Monoids}
For standard notation and terminology relating to monoids, we refer the reader to~\cite{Alm,Cli-Pre,Rho-Ste}.
Let $M$ a finite monoid. Let $a,b\in M$. We say that $a\R b$ if $aM = bM$, $a\eL b$ if $Ma = Mb$ and $a\HH b$ if $a\R b$ and $a\eL b$. Also, we say that $a\J b$, if $MaM = MbM$.
The relations $\R,\eL$, $\HH$ and $\J$ are Green relations and all of them are equivalence relations were first introduced by Green~\cite{Gre}.
We call $R_a,L_a,H_a$ and $J_a$, respectively, the $\R,\eL,\HH$ and $\J$-class containing $a$.
An important property of finite monoids is the stability property that $J_m\cap Mm = L_m$ and $J_m\cap mM = R_m$, for every $m \in M$.
For $\J$-classes $J_a$ and $J_b$, we can define the partial order $\leq$ as follows:
$$MaM\subseteq MbM\mbox{ if and only if }J_a\leq J_b.$$

 An element $e$ of $M$ is called idempotent if $e^2 = e$. The set of all idempotents of $M$ is denoted by $E(M)$; more generally, for any $X\subseteq M$, we put $E(X)=X\cap E(M)$.
An idempotent $e$ of $M$ is the identity of the monoid $eMe$. The group of units $G_e$ of $eMe$ is called the maximal subgroup of $M$ at $e$. Note that $G_e=H_e$.

An element $m$ of $M$ is called (von Neumann) regular if there exists an element $n\in M$ such that $mnm=m$. Note that an element $m$ is regular if and only if $m\eL e$, for some $e\in E(M)$, if and only if $m\R f$, for some $f\in E(M)$. A $\J$-class $J$ is regular if all its elements are regular, if and only if $J$ has an idempotent, if and only if $J^2\cap J\neq\emptyset$.  Note that if $N$ is a submonoid of $M$ and $a,b\in N$ are regular in $N$, then $a\K b$ in $N$ if and only if $a\K b$ in $M$
where $\K$ is any of $\R, \eL$ or $\HH$ (\cite[Proposition A.1.16]{Rho-Ste}).

Let $G$ be a group, $n$ and $m$ be integers and $P$ be an $m\times n$ matrix with entries in $G\cup\{0\}$.
The Rees matrix semigroup $\mathcal{M}^{0}(G, n,m;P)$ is the set of all triples $(i,g,j)$ where $g\in G$, $1\leq i \leq n$ and $1\leq j\leq m$, together with  $0$, and the following binary operation between non-zero elements
\begin{equation*}
(i,g,j)(i',g',j')  = \begin{cases}
  (i,gp_{ji'}g',j')& \text{if}\ p_{ji'}\neq 0;\\
  0& \text{otherwise},
\end{cases}
\end{equation*}
for every $(i,g,j),(i',g',j')\in \mathcal{M}^{0}(G, n,m;P)$ where $P=(p_{ij})$. The Rees matrix semigroup $\mathcal{M}^{0}(G, n,m;P)$ is regular if and only if each row and each column of $P$ contains a non-zero entry, in which case all non-zero elements are $\J$-equivalent.

\subsection{$M$-sets}
An $M$-set, for a monoid $M$, consists of a set $\Omega$ together with a mapping $M \times \Omega \rightarrow \Omega$, written $(m, \omega) \mapsto m\omega$ and
called an action, such that:
\begin{enumerate}
\item $1\omega = \omega$;
\item $m_2(m_1\omega) = (m_2m_1)\omega$,
\end{enumerate}
for every $\omega\in\Omega$ and $m_1,m_2\in M$.
The pair $(M,\Omega)$ is called a transformation monoid if $M$ acts faithfully on the $\Omega$. We write $T_{\Omega}$ for the full transformation monoid on $\Omega$, that is, the monoid of all self-maps of $\Omega$. Transformation monoids on $\Omega$ amount
to submonoids of $T_{\Omega}$.
We say that $M$ is transitive on $\Omega$ if $M\omega = \Omega$ for all $\omega\in\Omega$. The rank of $m \in M$ is defined by $\rk(m) = \abs{m\Omega}$.
The rank is constant on $\J$-classes and, for a transformation monoid, the minimum rank is attained precisely on the minimal ideal $I(M)$.

A non-empty subset $\Delta$ of $\Omega$ is $M$-invariant if $M\Delta \subseteq \Delta$.
The set $\Omega^2$ is a finite transformation monoid via $m(\alpha, \beta) = (m\alpha, m\beta)$, for every $(\alpha, \beta)\in\Omega^2$. Let $\Delta = \{(\alpha, \alpha) \in \Omega^2 \mid \alpha \in \Omega\}$. One says that $M$ acts $2$-transitively on $\Omega$, if for every $(\alpha,\beta),(\alpha',\beta')\in \Omega^2\setminus \Delta$ there exists an element $m\in M$ such that $m(\alpha,\beta)=(\alpha',\beta')$.
We also say that $M$ is $0$-transitive on $\Omega$, if for a necessarily unique $\omega_0\in \Omega$, $M\omega_0 = \{\omega_0\}$ and $M\omega =\Omega$, for all $\omega\in\Omega\setminus\{\omega_0\}$. Traditionally, one uses $0$ instead of $\omega_0$ for the fixed point of $M$ but to avoid confusion with the $0$ of associated vector spaces, we write $\omega_0$. But we still use the term "$0$-transitive".

A congruence on an $M$-set $\Omega$ is an equivalence relation $\equiv$ such that $\alpha\equiv\beta$ implies $m\alpha\equiv m\beta$ for all $\alpha,\beta \in\Omega$ and $m \in M$. In this case, the quotient $\Omega/{\equiv}$ becomes an $M$-set in the natural way and the quotient map $\Omega\rightarrow \Omega/{\equiv}$ is a morphism. A transformation monoid $(M,\Omega)$ is primitive if it admits no non-trivial proper congruences.
We refer the reader for more details on this concept to~\cite{Ben-Transformation-Monoids-2010}.

\subsection{Transformation modules and representations of monoids}
Let $(M,\Omega)$ be a finite transformation monoid and $\mathbb{F}$ a field. By extending the action of $M$ on $\Omega$ linearly, as the basis, $\mathbb{F}\Omega$ is a left $\mathbb{F}M$-module where $\mathbb{F}M$ is the monoid algebra of $M$ on $\mathbb{F}$. It is the transformation module associated with the action. Also, we have that
$\mathbb{F}^\Omega = \{f\colon  \Omega \rightarrow \mathbb{F}\}$ is a right $\mathbb{F}M$-module by putting $(fm)(\omega) = f(m\omega)$.
We identify $\mathbb{F}^\Omega$ as the dual space to $\mathbb{F}\Omega$ in the natural way.
The map $\eta\colon  \mathbb{F}\Omega \rightarrow \mathbb{F}$ sending each element of $\Omega$ to $1$ is called the augmentation map and hence we write $\ker \eta = \Aug(\mathbb{F}\Omega)$ which is the augmentation submodule of $\mathbb{F}\Omega$. Let $\mathcal{W}$ be an $\mathbb{F}M$-submodule of $\mathbb{F}^\Omega$. The $\mathbb{F}M$-submodule
$\mathcal{W}^{\bot}$ of $\mathbb{F}\Omega$ is the null space of $\mathcal{W}$ as follows:
$$\{v \in \mathbb{F}\Omega\mid w(v) = 0,\mbox{ for every }w\in\mathcal{W}\}.$$

We recall the map $1_B\in \mathbb{F}^\Omega$, for a subset $B\in \Omega$, defined as follows:
\begin{equation*}
1_B(x) = \begin{cases}
  1& \text{if}\ x\in B;\\
  0& \text{otherwise},
\end{cases}
\end{equation*}
for every $x\in\Omega$.

Let $S$ be a simple $\mathbb{F}M$-module. An idempotent $e \in E(M)$ is called an apex for $S$ if $eS \neq 0$ and $I_eS = 0$ where $I_e = eMe \setminus G_e$. Let us recall the notation $I(e) = \{m \in M \mid e \not\in MmM\}$.
If $S$ is simple with an apex $e$, then $I(e) = \{m \in M \mid mS = 0\}$ and $f\in E(M)$ is an apex for $S$ if and only if $MeM = MfM$ (\cite[Proposition 5.4]{Ben-Rep-Monoids-2016}). In general, every simple $\mathbb{F}M$-module has an apex (unique up to $\J$-equivalence) and there is a bijection between isomorphism classes of simple $\mathbb{F}M$-modules with apex $e \in E(M)$ and isomorphism classes of simple $\mathbb{F}G_e$-modules~\cite[Theorem 5.5]{Ben-Rep-Monoids-2016}.
If $V$ is a module over a ring $R$ and $X\subseteq V$, then $\langle X\rangle_R$ denotes the $R$-module generated by $X$.
We refer the reader for more details to~\cite{Hof-Kun,Ben-Rep-Monoids-2016}.

\subsection{Graphs and simplicial complexes}
A graph $\Gamma$ is a pair of a set $V(\Gamma)$ of vertices and a set $E(\Gamma)$ of unordered pairs from $V(\Gamma)$ called edges. If $e=\{v_1,v_2\}$, for some vertices $v_1,v_2\in V(\Gamma)$ and edge $e\in E(\Gamma)$, then we say that there is an edge between $v_1$ and $v_2$.  A graph $\Gamma'$ is a subgraph of $\Gamma$ if $V(\Gamma') \subseteq V(\Gamma)$ and $E(\Gamma') \subseteq E(\Gamma)$.
A path in the graph $\Gamma$ is a non-empty alternating sequence $v_0v_1 \ldots v_k$ of vertices in $\Gamma$ such that $\{v_{i},v_{i+1}\}\in E(\Gamma)$, for all $i<k$.
If $v_0=v_k$ then the path is a cycle.
The graph $\Gamma$ is connected if for every distinct vertices of $\Gamma$, there exists a path between them.  The degree of a vertex $v\in V(\Gamma)$ is the number of edges incident on $v$.
A graph $\Gamma$ is said to be a star graph of order $n$ if it is a tree on $n$ vertices with one vertex having degree $n-1$ and the other $n-1$ having degree $1$ and a vertex in the graph $\Gamma$ is said to be a star vertex if its degree is equal to $n-1$. A graph $\Gamma$ is said complete if for all distinct vertices $v_1,v_2\in V(\Gamma)$, there exists an edge between them.

A digraph graph $\Delta$ is a graph $(V(\Delta),E(\Delta))$ together with an orientation of each edge, that is two maps $\init\colon E(\Delta)\rightarrow V(\Delta)$ and $\ter\colon E(\Delta)\rightarrow V(\Delta)$ assigning to every edge $e$ an initial vertex $\init(e)$ and a terminal vertex
$\ter(e)$. A directed path in the digraph graph $\Delta$ is a non-empty alternating sequence $v_0v_1 \ldots v_k$ of vertices in $\Delta$ such that there exist edges $e_1,\ldots, e_k$ with $\init(e_{i}) = v_{i-1}$ and $\ter(e_{i})=v_{i}$, for all $1\leq i\leq k$.
If $v_0=v_k$ then the directed path is a directed cycle.
An acyclic digraph is a finite digraph with no directed cycles.
We refer the reader for more details for the concept in graph theory to~\cite{Bon-Mur,Die}.

Let $\Omega$ be a set. A simplicial complex $\mathcal{K}$ on the set $\Omega$ is a pair $(\Omega,\mathcal{F})$ such that the set $\mathcal{F} $ satisfies the following conditions:
\begin{enumerate}
\item $\mathcal{F}\subseteq \mathcal{P}(\Omega)$;
\item if $X\subseteq Y$ and $Y\in \mathcal{F}$, then $X\in \mathcal{F}$;
\item $\bigcup_{X\in \mathcal{F}} X=\Omega$.
\end{enumerate}
It is clear that $\mathcal{F}$ contains all subsets $\{\omega\}$, for $\omega\in\Omega$. The pair $(\Omega,\mathcal{P}(\Omega))$ is called a simplex.

Let $\dim(\mathcal{K})=\max\{\abs{X}-1\mid X\in \mathcal{F}\}$. We define the graph $\mathcal{K}^1$ as follows:
\begin{enumerate}
\item $V(\mathcal{K}^1)=\Omega$;
\item $\{\omega,\omega'\}\in E(\mathcal{K}^1)$ if and only if $\{\omega,\omega'\}\in \mathcal{F}$.
\end{enumerate}
We say that the simplicial complex $\mathcal{K}$ is connected if the graph $\mathcal{K}^1$ is connected.
If $(\Omega,\mathcal{F})$ and $(\Omega',\mathcal{F}')$ are simplicial complexes, a map $f\colon \Omega\rightarrow\Omega'$ is a simplicial map if $X\in \mathcal{F}$, then $f(X)\in \mathcal{F}'$. The image of a connected simplicial complex under a simplicial map is connected.

Let $\Gamma=(V(\Gamma), E(\Gamma))$ be a graph. The graph $\Gamma$ is a simplicial complex with $\dim(\Gamma)\leq 1$.
Also, for a map $f\colon V(\Gamma)\rightarrow V(\Gamma)$, the map $f$ is a simplicial map if $\{v_1,v_2\}\in E(\Gamma)$ then $f(v_1)=f(v_2)$ or $\{f(v_1),f(v_2)\}\in E(\Gamma)$. A simplicial map $f$ on a digraph preserves orientation if whether $e$ is an edge with $f(e)$ an edge, then $\init(f(e))=f(\init(e))$ and $\ter(f(e))=f(\ter(e))$.

\subsection{Set systems and incidence matrices}
A set system on $\Omega$ is a collection $\mathcal{E}$ of subsets of $\Omega$ and the incidence matrix of $\mathcal{E}$, $I(\mathcal{E})$, is the $\Omega\times\mathcal{E}$ matrix defined as follows:
\begin{equation*}
I(\mathcal{E})_{\omega,B}= \begin{cases}
 1& \text{if}\ \omega\in B; \\
 0& \text{otherwise},
\end{cases}
\end{equation*}
for every $\omega\in\Omega$ and $B\in \mathcal{E}$. It is known that random incidence matrices of set systems with $\abs{\Omega}=n=\abs{\mathcal{E}}$ are invertible with probability tending to $1$ as $n\rightarrow \infty$~\cite{Kom}. If $\{\alpha\}\in \mathcal{E}$, for all $\alpha\in\Omega$, then $\rank I(\mathcal{E})=\abs{\Omega}$.

\subsection{Posets}
A partial order is a binary relation $\leq$ over a set $P$ which is reflexive, antisymmetric, and transitive. A set with a partial order is called a partially ordered set or poset.  A subset $U$ of $P$ is an upper set, if $x\in U$ and $x<y$, then $y\in U$. A map $f$ between posets is order preserving if $\alpha\leq\beta$ then $f(\alpha)\leq f(\beta)$.


\section{The characterization of simple augmentation modules}

Throughout in this section we suppose that $(M,\Omega)$ is a finite transformation monoid which is not a group, $\abs{\Omega}>1$ and $\mathbb{F}$ a field.

\begin{lem}\label{tran-0tran}
If the augmentation submodule $\Aug(\mathbb{F}\Omega)$ is simple and $\abs{\Omega}>2$, then $M$ is primitive and one of the following conditions holds:
\begin{enumerate}
\item the monoid $M$ is transitive on $\Omega$;
\item the monoid $M$ is $0$-transitive on $\Omega$.
\end{enumerate}
\end{lem}

\begin{proof}
First, we prove that $M$ is primitive. Let $\equiv$ be a congruence on $\Omega$ and $V^{\equiv}=\langle \alpha-\beta\mid \alpha\equiv \beta\rangle_{\mathbb{F}}$. Since $\Aug(\mathbb{F}\Omega)$ is spanned by the differences $\alpha-\beta$ with $\alpha, \beta \in \Omega$ and $V^{\equiv}$ is equal to the kernel of the $\mathbb{F}M$-module homomorphism $\mathbb{F}\Omega\rightarrow\mathbb{F}[\Omega/{\equiv}]$, $V^{\equiv}$ is a submodule of $\Aug(\mathbb{F}\Omega)$. If the congruence $\equiv$ is non-trivial and non-universal then $V^{\equiv}$ is a proper and non-zero submodule of $\Aug(\mathbb{F}\Omega)$, a contradiction with the assumption that the augmentation submodule $\Aug(\mathbb{F}\Omega)$ is simple. Hence, $M$ is primitive. Now, by~\cite[Proposition 5.1]{Ben-Transformation-Monoids-2010}, the result follows.
\end{proof}

We recall that $I(M)$ denotes the unique minimal ideal of $M$.

\begin{lem}\label{eW}
Suppose that $I(M)$ consists of a constant map and $M\setminus I(M)$ has a unique minimal $\J$-class $J$ which is regular. Let $e\in E(J)$,
$$\mathcal{W}=\langle 1_B\mid B= f^{-1}f\omega, \mbox{ for some }\omega\in\Omega\mbox{ and }f\in E(J)\rangle_{\mathbb{F}}$$
and
$$\mathcal{W'}=\langle 1_B\mid B= e^{-1}e\omega, \mbox{ for some }\omega\in\Omega  \rangle_{\mathbb{F}M}.$$
Then, we have $\mathcal{W}=\mathcal{W'}$, $\mathcal{W}^{\bot}$ is a submodule of $\Aug(\mathbb{F}\Omega)$
and is the largest submodule $W$ of $\mathbb{F}\Omega$ with $eW=0$.
\end{lem}

\begin{proof}
There exist subsets $B_1,\ldots,B_r$ of $\Omega$ such that $\Omega$ is partitioned by them and there exists an element $\omega_i\in\Omega$ such that $B_i=e^{-1}e\omega_i$, for each $1\leq i\leq r$.
We prove that $\mathcal{W}^{\bot}=\mathcal{W'}^{\bot}$.
Let $w=\sum c_{\omega}\omega\in \mathcal{W}^{\bot}$.
Then, we have $1_{B_{i}}(w)=0$, for every $1\leq i\leq r$ and, thus, we have $\sum c_{\omega}=0$. In particular, we have $w\in \Aug(\mathbb{F}\Omega)$.
Let $1\leq i\leq r$ and $m\in M$. If $(em)^{-1}\omega_i=\emptyset$ then $m1_{B_{i}}(w)=1_{B_{i}}(mw)=0$. Now, suppose that $m1_{B_i}(\omega)\neq 0$.
If $em$ is a constant map, then $(em)^{-1}\omega_i=\Omega$ and, thus, $m1_{B_{i}}(w)=1_{B_{i}}(mw)=\sum c_{\omega}=0$. Otherwise, we have $em\in J$ because $J$ is the minimal $\J$-class of $M\setminus I(M)$.
There is an idempotent $f\in E(J)$ such that $f\eL em$.
Hence, we have $(em)^{-1}\omega_i=f^{-1}\omega_f$, for some $\omega_f\in f\Omega$.
Thus, we have $m1_{B_{i}}(w)=1_{B_{i}}(mw)=1_{(em)^{-1}\omega_i}(w)=1_{B_f}(w)=0$ because $1_{B_f}\in \mathcal{W}$ as $B_f=f^{-1}\omega_f$. It follows that $\mathcal{W}^{\bot} \subseteq\mathcal{W'}^{\bot}$.
Now, suppose that $w\in \mathcal{W'}^{\bot}$, $f\in E(J)$ and $\omega_f\in f\Omega$. There exists an element $m\in J$ such that $f\eL m$ and $m\R e$. Hence, we have $f^{-1}\omega_f=m^{-1}\omega_m$, for some $\omega_m\in m\Omega$, and $\omega_m=\omega_i$, for some integer $1\leq i\leq r$, since $em=m$.
Now, as $em=m$ and $f^{-1}\omega_f=m^{-1}\omega_m=(em)^{-1}\omega_i=m^{-1}e^{-1}\omega_i=m^{-1}B_i$, we have $1_{B_{f}}(w)=1_{m^{-1}B_i}(w)=m1_{B_i}(w)=0$ where $B_f=f^{-1}\omega_f$.
It follows that $\mathcal{W}^{\bot}=\mathcal{W'}^{\bot}$ and, thus, we have $\mathcal{W}=\mathcal{W'}$.

Suppose that $W$ is a submodule of $\mathbb{F}\Omega$ with $eW=0$. Let $w=\sum c_{\omega}\omega\in W$ and $m\in M$.
We have $emw=0$.
Since
\begin{eqnarray}\label{emw}
emw=em\sum c_{\omega}\omega=\sum_{1}^{r}(\sum_{m\omega\in B_i}c_{\omega})\omega_i,
\end{eqnarray}
we have $\sum_{m\omega\in B_i}c_{\omega}=0$, for all $1\leq i\leq r$. It follows that $1_{B_i}(mw)=0$ and, thus, $m1_{B_i}(w)=0$, for every $1\leq i\leq r$. Therefore, we have $w\in\mathcal{W'}^{\bot}$ and, thus, $w\in\mathcal{W}^{\bot}$.
Again as $\mathcal{W}^{\bot}=\mathcal{W'}^{\bot}$, if $v\in \mathcal{W}^{\bot}$, then $1_{B_i}(v)=0$, for every $1\leq i\leq r$, and, thus, $ev=0$; see~(\ref{emw}) with $m=1$ and $w=v$. Therefore, we have $e\mathcal{W}^{\bot}=0$.
Therefore, the largest submodule $W$ of $\mathbb{F}\Omega$ with $eW=0$ is $\mathcal{W}^{\bot}$.
\end{proof}

For a transformation monoid $M\leq T_{\Omega}$ such that $I(M)$ contains a constant map and $M\setminus I(M)$ has a unique minimal $\J$-class $J$, which, moreover, is regular, we can define a graph $\Gamma(M)=(\Omega,E)$ where $$E=\{\{v_1,v_2\}\mid fv_1=v_1\mbox{ and }fv_2=v_2, \mbox{ for some } f\in E(J)\}.$$ Also, we define the surjective map $\phi_M\colon \mathbb{F}\Omega\rightarrow \mathbb{F}\pi_0(\Gamma(M))$ with $\phi_M(\alpha)$ the connected component of $\alpha$ in $\Gamma(M)$, where $\pi_0(\Gamma(M))$ denotes the set of connected components of $\Gamma(M)$.

\begin{lem}\label{GraphGamma}
Suppose that $I(M)$ consists of a constant map and $M\setminus I(M)$ has a unique minimal $\J$-class $J$ which is regular.
The following conditions hold:
\begin{enumerate}
\item the monoid $M$ acts on $\Gamma(M)$ by simplicial maps;
\item the map $\phi_M$ is an $\mathbb{F}M$-module homomorphism.
\end{enumerate}
\end{lem}

\begin{proof}
Let $\alpha,\beta\in \Omega$ and $m\in M$. Suppose that there is an edge between $\alpha$ and $\beta$ in $\Gamma(M)$. Hence, there is an idempotent $h\in E(J)$ such that $h\alpha=\alpha$ and $h\beta=\beta$. We prove that if $m\alpha\neq m\beta$ then there is an edge between $m\alpha$ and $m\beta$ in $\Gamma(M)$. Hence, we suppose that $m\alpha\neq m\beta$. Since $m\alpha=mh\alpha$, $m\beta= mh\beta$ and $m\alpha\neq m\beta$, we have $mh\in J$ and, thus, there is an idempotent $h'\in J$ such that $h'mh=mh$. It follows that $h'm\alpha=m\alpha$ and $h'm\beta=m\beta$ and, thus, there is an edge between $m\alpha$ and $m\beta$. Hence, $M$ acts on $\Gamma(M)$ by simplicial maps.
It follows that the map $\phi_M$ is an $\mathbb{F}M$-module homomorphism as simplicial maps preserve connected components.
\end{proof}

We are now prepared to prove our characterization of simple augmentation modules.

\begin{thm}\label{aug-simple-monoids}
Let $M\leq T_{\Omega}$ with $M$ not a group.
The augmentation submodule $\Aug(\mathbb{F}\Omega)$ is simple if and only if the following conditions hold:
\begin{enumerate}
\item the monoid $M$ contains a constant map;
\item the subset $M\setminus I(M)$ has a unique minimal $\J$-class $J$ and moreover $J$ is regular;
\item if $e\in E(J)$, then $\Aug(\mathbb{F}e\Omega)$ is a simple $\mathbb{F}G_e$-module;
\item the rank of the incidence matrix of the set system $$\{B\mid B= f^{-1}f\omega, \mbox{ for some }\omega\in\Omega\mbox{ and }f\in E(J)\}$$ is $\abs{\Omega}$ over $\mathbb{F}$;
\item the graph $\Gamma(M)$ is connected.
\end{enumerate}
\end{thm}

\begin{proof}
First, we suppose that the augmentation submodule $\Aug(\mathbb{F}\Omega)$ is simple.
If $\abs{\Omega}=2$ and $M$ has no constant map, then $M$ is a group, a contradiction with the assumption. Now, suppose that $\abs{\Omega}>2$. By Lemma~\ref{tran-0tran}, the monoid $M$ is transitive or $0$-transitive on $\Omega$. By~\cite[Theorem 13.6]{Ben-Rep-Monoids-2016}, the former case implies that the ideal $I(M)$ is the set of all constant mappings on $\Omega$ and the latter case implies that $I(M)$ consists of the constant map to the unique fixed point of $M$. Therefore the monoid $M$ has a constant map.
Since the submodule $\Aug(\mathbb{F}\Omega)$ is simple, it has an apex $e$ in a regular $\J$-class $J$.
By~\cite[Proposition 7.9]{Ben-Transformation-Monoids-2010} and since $I(M)$ consists of constant maps, we have $I(M) = \{m \in M \mid m \Aug(\mathbb{F}\Omega) = 0\}$. Hence, $I(e)=I(M)$ and, thus, the $\J$-class $J$ is unique and minimal in $M\setminus I(M)$. Now, by~\cite[Theorem 5.5(i)]{Ben-Rep-Monoids-2016} and~\cite[Proposition 5.4(ii)]{Ben-Rep-Monoids-2016}, condition (3) holds.
Let $e\in E(J)$ and $\mathcal{W}=\langle 1_B\mid B= f^{-1}f\omega, \mbox{ for some }\omega\in\Omega\mbox{ and }f\in E(J)  \rangle_{\mathbb{F}}$. By Lemma~\ref{eW}, $\mathcal{W}^{\bot}$ is a submodule of $\Aug(\mathbb{F}\Omega)$. The augmentation submodule $\Aug(\mathbb{F}\Omega)$ is simple. Hence, $\mathcal{W}^{\bot}=0$ or $\mathcal{W}^{\bot}=\Aug(\mathbb{F}\Omega)$.
Now, as $e\mathcal{W}^{\bot}=0$ and $e\Aug(\mathbb{F}\Omega)\neq 0$, we have $\mathcal{W}^{\bot}=0$. If follows that $\mathcal{W}=\mathbb{F}^{\Omega}$ and, thus, condition (4) holds. By Lemma~\ref{GraphGamma}, the surjective map $\phi_M$ is an $\mathbb{F}M$-module homomorphism. It is clear that
$$
\ker\phi_M=\langle \alpha-\alpha'\mid\alpha \mbox{ and }\alpha'\mbox{ are in the same connected component}\rangle_{\mathbb{F}}.
$$
Since $2\leq \rk(J)$, the graph $\Gamma(M)$ has edges and, thus, we have $\ker\phi_M\neq\{0\}$. Now, as the submodule $\Aug(\mathbb{F}\Omega)$ is simple and $\ker\phi_M\subseteq \Aug(\mathbb{F}\Omega)$ is a submodule, we have $\ker\phi_M= \Aug(\mathbb{F}\Omega)$. It follows $\abs{\pi_0(\Gamma(M)}=1$ and so the graph $\Gamma(M)$ is connected.

Now, we suppose that $(M,\Omega)$ satisfies the five conditions of the theorem.
Suppose that $W$ is a non-zero submodule of $\Aug(\mathbb{F}\Omega)$.
Let $e\in E(J)$.
By Lemma~\ref{eW} the largest submodule $V$ of $\mathbb{F}\Omega$ with $eV=0$ is
$$\mathcal{W}^{\bot}=\langle 1_B\mid B= f^{-1}f\omega, \mbox{ for some }\omega\in\Omega\mbox{ and }f\in E(J)  \rangle_{\mathbb{F}}^{\bot}.$$ Now, since by condition (4), $\mathcal{W}=\mathbb{F}^{\Omega}$,
we have $\mathcal{W}^{\bot}=0$ and so $eW\neq 0$. By condition (3), the augmentation submodule $\Aug(\mathbb{F}e\Omega)=e\Aug(\mathbb{F}\Omega)$ is a simple $\mathbb{F}G_e$-module. Hence, we have $eW=\Aug(\mathbb{F}e\Omega)$ and, thus, we have $\alpha-\alpha'\in eW$, for every $\alpha,\alpha'\in e\Omega$.
Let $f\in E(J)$, $\omega_1\neq\omega_2\in \im f$ and choose $m\in M$ such that $f\R m$ and $m\eL e$. Since $\omega_1,\omega_2\in \im f=\im m$, there exist elements $\alpha_1,\alpha_2$ such that $m\alpha_1=\omega_1$ and $m\alpha_2=\omega_2$. As $\omega_1\neq \omega_2$ and $m\eL e$,
we have $e\alpha_1\neq e\alpha_2$ and thus, the element $e\alpha_1- e\alpha_2$ is a non-zero element of $e\Aug(\mathbb{F}\Omega)$. Now, as $\omega_1-\omega_2=m(e\alpha_1- e\alpha_2)$, we have $\omega_1-\omega_2\in \langle eW\rangle_{\mathbb{F}M}$. Let $\lambda$ and $\lambda'$ be distinct elements of $\Omega$.
By condition (5), the graph $\Gamma(M)$ is connected.  Hence, there exist elements $\lambda_1,\ldots,\lambda_s\in \Omega$ such that $\lambda_1=\lambda$, $\lambda_s=\lambda'$ and there is an edge between $\lambda_i$ and $\lambda_{i+1}$ in the graph $\Gamma(M)$, for all $1\leq i\leq t-1$. By above, we have $\lambda_i-\lambda_{i+1}\in \langle eW\rangle_{\mathbb{F}M}$, for all $1\leq i\leq t-1$.
Therefore, $\lambda-\lambda'\in \langle eW\rangle_{\mathbb{F}M}$. It follows that $\Aug(\mathbb{F}\Omega)=\langle eW\rangle_{\mathbb{F}M}\subseteq W$.
\end{proof}

\begin{rmk}
The necessity of condition (5) in Theorem~\ref{aug-simple-monoids} can be understood in the following way more conceptual way.  Since $M$ acts by simplicial maps on $\Gamma(M)$, the augmented simplicial chain complex for $\Gamma(M)$ with coefficients in $\mathbb F$ is a chain complex of $\mathbb FM$-modules.  Therefore, the reduced homology $\widetilde{H}_0(\Gamma(M))$ is a quotient $\mathbb FM$-module of $\Aug(\mathbb F\Omega)$ and it is a proper quotient because $\Gamma(M)$ has edges.  By simplicity of $\Aug(\mathbb F\Omega)$, it must be $0$ and so $\Gamma(M)$ must be connected.
\end{rmk}

To clarify the connection with $2$-transitivity we establish the following proposition.

\begin{prop}\label{2-tran}
If $(M,\Omega)$ is $2$-transitive and $M$ is not a group, then $M$ satisfies conditions (1), (2) and (5) of Theorem~\ref{aug-simple-monoids}, $M$ satisfies condition (3) for the field $\mathbb{C}$ and, moreover, the graph $\Gamma(M)$ is complete.
\end{prop}

\begin{proof}
Suppose that $M$ has no constant map. Then $I(M)$ has an idempotent $a\in I(M)$ with $a$ not a constant map. Since $M$ is not a group, the element $a$ is not the identity of $M$ and, thus, there exist distinct $\alpha_1,\alpha_2\in \Omega$ such that $a\alpha_1=\alpha_1$ and $a\alpha_2=\alpha_1$. Also, since $a$ is not a constant map, there exists $\alpha_3\in\Omega$ such that $a\alpha_3=\alpha_3$ and $\alpha_1\neq\alpha_3$. Since $M$ is $2$-transitive, there exists an element $b$ such that $b\alpha_1=\alpha_2$ and $b\alpha_3=\alpha_1$.
Now, as $aba\alpha_1=aba\alpha_3=\alpha_1$, the rank of the element $aba$ is strictly smaller than the rank of $a$, a contradiction with the assumption that $a\in I(M)$. Therefore, $M$ has a constant map. So $I(M)$ consists of a constant map.

Now, suppose that  $J$ and $J'$ are minimal $\J$-classes of $M\setminus I(M)$. Let $m_1\in J$, $m_2\in J'$. Then, there exist elements $\beta_i\in\Omega$ such that $m_1\beta_1=\beta_2$, $m_1\beta_3=\beta_4$, $m_2\beta_5=\beta_6$, $m_2\beta_7=\beta_8$, $\beta_2\neq\beta_4$ and $\beta_6\neq\beta_8$. Since $M$ is $2$-transitive and $\beta_5\neq \beta_7$, there exists an element $m\in M$ such that $m\beta_2=\beta_5$, $m\beta_4=\beta_7$. Since $m_2mm_1\beta_1\neq m_2mm_1\beta_3$, $m_2mm_1\not\in I(M)$ and thus $J=J'$. Therefore, $M\setminus I(M)$ has a unique minimal $\J$-class $J$. Also, as $m_2mm_1\in J$, $J$ is a regular $\J$-class.

We now prove $G_e$ is $2$-transitive on $e\Omega$.
Let $e\in E(J)$ and $\gamma_i\in e\Omega$ with $\gamma_1\neq\gamma_2$ and $\gamma_3\neq\gamma_4$. Since $M$ is $2$-transitive, there exists an element $m\in M$ such that $m\gamma_1=\gamma_3$ and $m\gamma_2=\gamma_4$. Since $e$ is idempotent, we have $e\gamma_i=\gamma_i$, for all integers $1\leq i\leq 4$. It follows that $eme\gamma_1=\gamma_3$ and $eme\gamma_2=\gamma_4$. Thus, $eme\not\in I(M)$ and so $eme\in eMe\cap J=G_e$. Therefore, the maximal subgroup $G_e$ is $2$-transitive on $e\Omega$. Now, by~\cite[Proposition B.12]{Ben-Rep-Monoids-2016}, $\Aug(\mathbb{C}e\Omega)$ is a simple $\mathbb{C}G_e$-module.

Let $\alpha,\beta\in\Omega$ with $\alpha\neq\beta$. Since the rank of $J$ is more than one, there is an element $a$ in $J$ such that $a\gamma_1=\gamma_3$ and $a\gamma_2=\gamma_4$, for some elements $\gamma_i\in \Omega$, for all $1\leq i\leq 4$, with $\gamma_3\neq \gamma_4$ whence $\gamma_1\neq \gamma_2$. Now, as $M$ is $2$-transitive, there are elements $m_1$ and $m_2$ in $M$ such that $m_1\alpha=\gamma_1$, $m_1\beta=\gamma_2$, $m_2\gamma_3=\alpha$ and $m_2\gamma_4=\beta$. Hence, $m_2am_1\alpha=\alpha$ and $m_2am_1\beta=\beta$. As the rank of $m_2am_1$ is more than one, $a\in J$ and $J$ is the unique minimal $\J$-class of $M\setminus I(M)$, we have $m_2am_1\in J$. Hence, there is an idempotent $e\in J$ with $e\R m_2am_1$. So $e\alpha=\alpha$ and $e\beta=\beta$.
Therefore, the graph $\Gamma(M)$ is complete.
\end{proof}


\section{Examples and counterexamples}

In this section we investigate number of examples satisfying the five conditions of Theorem~\ref{aug-simple-monoids} and a having simple augmentation modules over any field.

\begin{prop}\label{GraGa}
Let $\Gamma=(\Omega, E(\Gamma))$ be a connected graph with $2\leq\abs{\Omega}$ and $\End(\Gamma)=\{f\colon \Omega\rightarrow\Omega\mid f\mbox{ is a simplicial map on }\Gamma\}$. Let $M\leq \End(\Gamma)$ be a transformation monoid such that $M$ satisfies the following conditions:
\begin{enumerate}
\item $M$ is transitive on $\Omega$;
\item $M$ is edge transitive (for all $\{v_1,v_2\},\{w_1,w_2\}\in E(\Gamma)$, there exists an element $m\in M$ such that $\{m(v_1),m(v_2)\}=\{w_1,w_2\}$);
\item for some edge $\{v_1,v_2\}\in E(\Gamma)$, there exists an element $m\in M$ such that $m(v_1)=m(v_2)$;
\item there exists an element $a\in M$ such that the rank of $a$ is $2$.
\end{enumerate}
Then $I(M)$ is the set of all constant maps of $M$ and $M\setminus I(M)$ has a unique minimal $\J$-class $J$ which is regular, consists of all elements of $M$ with rank two and satisfies conditions (3) and (5) of Theorem~\ref{aug-simple-monoids} and moreover $\Gamma(M)=\Gamma$.
\end{prop}

\begin{proof}
Let $f\in I(M)$. Suppose that $f$ is not a constant map. Then, we have $2\leq\abs{\im f}$. Since the graph $\Gamma$ is connected, there exist
elements $w_1,w_2\in \Omega$ such that $\{w_1,w_2\}\in E(\Gamma)$ and $w_1,w_2\in \im f$. By (2), there is an element $g\in M$ such that $\{g(w_1),g(w_2)\}=\{v_1,v_2\}$ with $v_1$ and $v_2$ as in (3). Hence, the rank of $mgf$ is strictly smaller than the rank of $f$. This is a contradiction with the assumption that $f\in I(M)$. Therefore, $I(M)$ is the set of all constant maps (using (1)).

Let $J=\{j\in M\mid  \rk(j)=2\}$. By (4), the set $J$ is non-empty. We prove that $J$ is the unique minimal $\J$-class of $M\setminus I(M)$ and is regular. Let $f\in J$ and $m\in M\setminus I(M)$. Since $m\not \in I(M)$ and the graph $\Gamma$ is connected, there exists an edge $\{u_1,u_2\}\in E(\Gamma)$ such that $m(u_1)\neq m(u_2)$. Let $\im f=\{w_1,w_2\}$. Since $f\not \in I(M)$ and the graph $\Gamma$ is connected, there is an edge between $w_1$ and $w_2$. By (2), there exist elements $g,h\in M$ such that $\{g(w_1),g(w_2)\}=\{u_1,u_2\}$ and $\{h(m(u_1)),h(m(u_2))\}=\{w_1,w_2\}$. Therefore, we have $\im hmgf=\{w_1,w_2\}$ and $\ker f=\ker hmgf$. Thus, $f\HH hmgf$ in $T_{\Omega}$.
If $f$ is regular then $f\HH hmgf$ in $M$ and thus, $J$ is a minimal $\J$-class of $M\setminus I(M)$ which is regular. Now, we prove that $f$ is regular. Since the graph $\Gamma$ is connected, there exists an edge $\{z_1,z_2\}\in E(\Gamma)$ such that $\{f(z_1),f(z_2)\}=\{w_1,w_2\}$. By (2), there exists an element $k\in M$ such that $\{k(w_1),k(w_2)\}=\{z_1,z_2\}$. Hence, we have $\{fk(w_1),fk(w_2)\}=\{w_1,w_2\}$. It follows that the element $n=(fk)^2$ is an idempotent. Now, as $f=nf$, we have $f\R n$ and, so, $f$ is regular.

Let $e\in E(J)$. Since the rank of $e$ is $2$, $M$ satisfies condition (3) of Theorem~\ref{aug-simple-monoids}.
Now, suppose that $\{s_1,s_2\}\in E(\Gamma)$, for some elements $s_1,s_2\in \Omega$, and let $\{w_1,w_2\}=\im f$ with $f\in J$. Note that $\{w_1,w_2\}\in E(\Gamma)$ since $\Gamma$ is connected. By (2), there is $m$ in $M$ such that $\{m(w_1),m(w_2)\}=\{s_1,s_2\}$. As the rank of the element $mf$ is $2$, we have $mf\in J$. Hence, we have $\{s_1,s_2\}\in E(\Gamma(M))$ because $mf\R e$, for some $e\in E(J)$, and $emfw_i=s_i$. It follows that the graph $\Gamma$ is a subgraph of the graph $\Gamma(M)$. As the graph $\Gamma$ is connected, the graph $\Gamma(M)$ is connected, too.
Hence, $M$ satisfies condition (5) of Theorem~\ref{aug-simple-monoids}.
Now, we prove that $\Gamma(M)=\Gamma$. 	
If $f\in E(J)$, then $\im f$ is connected, as $\Gamma$ is connected. So, if $\im f=\{s_1,s_2\}$, for some elements $s_1,s_2\in \Omega$, then $\{s_1,s_2\}\in E(\Gamma)$. Hence, the graph $\Gamma(M)$ is a subgraph of the graph $\Gamma$.
\end{proof}

\begin{thm}\label{simplicial-complex-GraGa}
Let $\mathcal{K}=(\Omega,\mathcal{F})$ be a connected simplicial complex with $2\leq\abs{\Omega}$ and let $M=\End(\mathcal{K})$.
The following items hold:
\begin{enumerate}
\item The augmentation submodule $\Aug(\mathbb{F}\Omega)$ is simple, for every field $\mathbb{F}$.
\item We have $\Gamma(M)=\mathcal{K}^1$.
\item If $M$ is $2$-transitive, then the graph $\mathcal{K}^1$ is complete.
\end{enumerate}
\end{thm}

\begin{proof}
(1) Since $M=\End(\mathcal{K})$, we have $M\leq\End(\mathcal{K}^1)$.
Obviously, $I(M)$ consists of all constant maps and hence $M$ is transitive on $\Omega$ and satisfies (1) and (3) of Proposition~\ref{GraGa}.

Let $z_1,z_2,z_3\in \Omega$ and suppose that there is an edge between $z_1$ and $z_2$ in the graph $\mathcal{K}^1$.
We define the map $f_{z_1,z_2,z_3}$, for every $x\in\Omega$, as follows:
\begin{equation*}
f_{z_1,z_2,z_3}(x)= \begin{cases}
 z_1& \text{if}\ x=z_3;\\
 z_2& \text{if}\ x\in \Omega\setminus\{z_3\}.
\end{cases}
\end{equation*}
Obviously $f_{z_1,z_2,z_3}\in M$, since its image is a simplex. Now, let $\{v_1,v_2\},\{w_1,w_2\}\in E(\mathcal{K}^1)$. We have $$\{f_{w_1,w_2,v_1}(v_1),f_{w_1,w_2,v_1}(v_2)\}=\{w_1,w_2\}.$$
Thus $M$ is edge transitive. It is clear that the rank of $f_{w,v,v}$ is $2$, for every $v,w\in\Omega$ for which there is an edge between $v$ and $w$. Therefore by Proposition~\ref{GraGa}, $I(M)$ is the set of all constant maps of $M$ and $M\setminus I(M)$ has a minimal $\J$-class $J$ which is regular, contains all elements of $M$ with rank two and satisfies conditions (3) and (5) of Theorem~\ref{aug-simple-monoids}.
Let $v_1\in\Omega$. Since $\mathcal{K}$ is connected, there exists $\{v_1,v_2\}\in E(\mathcal{K}^1)$. Thus, $f_{v_1,v_2,v_1}\in E(J)$ and $f^{-1}_{v_1,v_2,v_1}(v_1)=\{v_1\}$. Hence, monoid $M$ satisfies condition (4) of the Theorem~\ref{aug-simple-monoids}. Therefore the augmentation submodule $\Aug(\mathbb{F}\Omega)$ is simple, for every field $\mathbb{F}$.

(2) In part (1), we showed that $M$ satisfies conditions of Proposition~\ref{GraGa} with $\Gamma=\mathcal{K}^1$. Hence, we have $\Gamma(M)=\mathcal{K}^1$.

(3) By part (2), we have $\Gamma(M)=\mathcal{K}^1$ Now, as $M$ is $2$-transitive, by Proposition~\ref{2-tran}, the graph $\mathcal{K}^1$ is complete.
\end{proof}

In this paper, by the endomorphism monoid of a graph we always mean its endomorphism monoid as a simplicial complex.

\begin{cor}
Let $\Gamma=(\Omega, E(\Gamma))$ be a connected graph with $2\leq\abs{\Omega}$ and let $M=\End(\Gamma)$.
The following items hold:
\begin{enumerate}
\item The augmentation submodule $\Aug(\mathbb{F}\Omega)$ is simple, for every field $\mathbb{F}$.
\item Suppose that $M'\leq T_{\Omega}$ satisfies conditions (1), (2) of Theorem~\ref{aug-simple-monoids} and $\Gamma(M')=\Gamma$. Then, $M'$ is a submonoid of $\End(\Gamma)$.
\end{enumerate}
\end{cor}

\begin{proof}
(1) Since the graph $\Gamma$ is a simplicial complex, by Theorem~\ref{simplicial-complex-GraGa}, the augmentation submodule $\Aug(\mathbb{F}\Omega)$ is simple, for every field $\mathbb{F}$.

(2) By Lemma~\ref{GraphGamma}, the monoid $M'$ acts on $\Gamma(M')$ by simplicial maps. So $M'\leq \End(\Gamma)$.
Now, by Lemma~\ref{GraphGamma}, since $\Gamma(M)=\Gamma$, $M'$ is a submonoid of $M$.
\end{proof}

In other words, transformation monoids of the form $\End(\Gamma)$ with $\Gamma$ a connected graph are universal amongst those with simple augmentation. Since for $3\leq n$, there are connected graphs that are not complete, this shows there are many transitive transformation monoids with simple augmentation over $\mathbb{C}$ that are not $2$-transitive.

\begin{prop}\label{GraGa2}
Suppose that $I(M)$ is the set of all constant maps of $M$ and $M\setminus I(M)$ has a minimal $\J$-class $J$ which is regular.
If there is a partial order on the set $\Omega$ and a function $$\phi\colon \Omega\rightarrow \{ B\mid B= f^{-1}f(\omega), \mbox{ for some }\omega\in\Omega\mbox{ and }f\in E(J)\}$$ such that $\alpha$ is the unique minimum element of $\phi(\alpha)$, for every $\alpha\in\Omega$, then $M$ satisfies condition (4) of Theorem~\ref{aug-simple-monoids}.
\end{prop}

\begin{proof}
Let $F=\{B\mid B= f^{-1}f(\omega), \mbox{ for some }\omega\in\Omega\mbox{ and }f\in E(J)\}$.
Suppose that there is a partial order $\leq$ on the set $\Omega$ and a function $\phi\colon \Omega\rightarrow F$ such that $\alpha$ is the unique minimum element of $\phi(\alpha)$, for every $\alpha\in\Omega$.
Order $\Omega=\{\omega_1,\ldots,\omega_{\abs{\Omega}}\}$ so that if $\omega_i\leq \omega_j$ then $i\leq j$. Then the incidence matrix $\Omega\times F$ has a lower triangular submatrix with $\abs{\Omega}$ columns and diagonal entries equal to one by considering the columns in the image of $\phi$ ordered $\phi(\omega_1),\ldots,\phi(\omega_{\abs{\Omega}})$. Thus the rank of the incidence matrix of $F$ is $\abs{\Omega}$ over $\mathbb{F}$.
It follows that $M$ satisfies condition (4) of Theorem~\ref{aug-simple-monoids}.
\end{proof}

\begin{thm}\label{Digraph-Simplicity}
Let $\Delta=(\Omega, E(\Delta))$ be an acyclic connected digraph with $2\leq\abs{\Omega}$ and
\begin{align*}
M=\End(\Delta)=\{f\colon \Delta\rightarrow\Delta\mid~& f\mbox{ is a simplicial map on }\Delta \mbox{ and } f \mbox{ preserves }\\
&\mbox{ orientation }\}.
\end{align*}
Thus the augmentation submodule $\Aug(\mathbb{F}\Omega)$ is simple, for every field $\mathbb{F}$.
\end{thm}

\begin{proof}
Since the graph $\Delta$ is acyclic, we can define a partial order
on $\Omega$ by $\omega<\omega'$ if there is a directed path $\omega$ to $\omega'$ in the digraph $\Delta$.
Let $\alpha,\beta\in \Omega$ with a directed edge from $\alpha$ to $\beta$ and $U$ be an upper set of $\Omega$ with $U\neq \Omega$. We define the map $f_{\alpha,\beta,U}\colon \Omega\rightarrow\Omega$, for every $x\in\Omega$, as follows:
\begin{equation*}
 f_{\alpha,\beta,U}(x)= \begin{cases}
 \beta& \text{if}\ x\in U;\\
 \alpha& \text{if}\ x\not\in U.
\end{cases}
\end{equation*}
It is clear that $f_{\alpha,\beta,U}\in M$.
Also, it is clear that $M$ contains all constant maps and, thus, $M$ satisfies (1) and (3) of Proposition~\ref{GraGa} for $\Gamma$ the underlying graph of $\Delta$.
Suppose that there are directed edges $v_1$ to $v_2$ and $w_1$ to $w_2$.
Let $U=\{w\in\Omega\mid v_2\leq w\}$.
We have $U \subsetneq\Omega$ is an upper set, $v_1\not\in U$ and $v_2\in U$.
Hence, $$f_{w_1,w_2,U}(v_1)=w_1\mbox{ and }f_{w_1,w_2,U}(v_2)=w_2.$$
Thus $M$ is edge transitive.
It is clear that the rank of $f_{v,w,U}$ is $2$, for every $v,w\in\Omega$ with a directed edge $v$ to $w$ and upper set $U \subsetneq\Omega$. Therefore, $M$ satisfies all conditions of Proposition~\ref{GraGa} and, thus, $I(M)$ is the set of all constant maps of $M$ and $M\setminus I(M)$ has a minimal $\J$-class $J$ which is regular, contains all elements of $M$ with rank two and satisfies conditions (3) and (5) of Theorem~\ref{aug-simple-monoids}.
Let $\alpha\in\Omega$.
If $\alpha$ is the unique minimum element of $\Omega$, then the element $\alpha$ is minimum in $f^{-1}(f(\alpha))$ for any $f\in E(J)$.
Assume $\alpha$ is not a unique minimum. There are two cases. If $\alpha$ is minimal, then since $\Delta$ is connected, there exists an element $\beta\in\Omega$ with a directed edge $\alpha$ to $\beta$. Let $U=\Omega\setminus\{\alpha\}$. Then $U$ is an upper set, $f_{\alpha,\beta,U}$ is an idempotent and $f_{\alpha,\beta,U}^{-1}(\alpha)=\{\alpha\}$ has $\alpha$ as a unique minimum. The second case is there exists an element $\beta'\in\Omega$ such that there is a directed edge $\beta'$ to $\alpha$. Hence, we have $f_{\beta',\alpha,\alpha^\uparrow}\in E(J)$ where $\alpha^\uparrow=\{\gamma\mid \alpha\leq \gamma\}$. It follows that $\alpha$ is the unique minimum element of the set $f_{\beta',\alpha,\alpha^\uparrow}^{-1}(\alpha)$.
Therefore, the partial order $\leq$ satisfies conditions of Proposition~\ref{GraGa2} and, thus, the augmentation submodule $\Aug(\mathbb{F}\Omega)$ is simple, for every field $\mathbb{F}$.
\end{proof}

An important special case is that of a poset.

\begin{cor}\label{connected-Hasse-diagram}
Let $(\Omega,\leq)$ be a finite poset
with a connected Hasse diagram and let $M$ be the transformation monoid on $\Omega$ as follows:
$$M=\{f\colon \Omega\rightarrow\Omega \mid\mbox{if }\alpha\leq\beta,\mbox{ then }f(\alpha)\leq f(\beta),\mbox{ for every }\alpha,\beta\in\Omega\}.$$
The augmentation submodule $\Aug(\mathbb{F}\Omega)$ is simple, for every field $\mathbb{F}$.
\end{cor}

\begin{proof}
We define a digraph $\Delta$ as follows:
\begin{enumerate}
\item $V(\Delta)=\Omega$;
\item there is an edge from $\omega$ to $\omega'$ in the graph $\Delta$ if and only if $\omega<\omega'$.
\end{enumerate}
Then, $\Delta$ contains the Hasse diagram of $\Omega$ as a subgraph and so $\Delta$ is connected. Also, since $(\Omega,\leq)$ is a poset, $\Delta$ is acyclic. Clearly, $M=\End(\Delta)$. Now, by Theorem~\ref{Digraph-Simplicity}, the result follows.
\end{proof}

Note that the category of finite posets is equivalent to the category of finite $T_0$ topological spaces and that connected posets correspond to connected $T_0$ spaces. So Corollary~\ref{connected-Hasse-diagram} can be reinterpreted as saying the monoid of continuous self-maps of a connected finite $T_0$ space has simple augmentation.

Our next construction shows (4) and (5) of Theorem~\ref{aug-simple-monoids} are independent.
Let $\Omega=\{1,\ldots,n\}$, $\Gamma$ be a connected graph with the vertex set $\Omega$ and $A$ be an $n\times r$ matrix over $\{0,1\}$ with no zero columns or all ones columns.
We define the set of mappings
\begin{align*}
F_{\Gamma,A}=\{f\colon \Omega\rightarrow\Omega\mid&\im f\mbox { is an edge of }\Gamma\mbox{ and }\ker f=\{A_j,\overline{A}_j\},\\
&\mbox{ for some integer } 1\leq j\leq r\}
\end{align*}
and the monoid
$M_{\Gamma,A}=F_{\Gamma,A}\cup C\cup \{1_{\Omega}\}\in \End(\Gamma)$ where $A_j=\{\alpha\in \Omega\mid A_{\alpha j}=1\}$, $\overline{A}_j=\Omega\setminus A_j$ and $C$ is the set of all constant maps on $\Omega$.
Let $f,g\in F_{\Gamma,A}$. If $fg\not\in C$, then $\im fg=\im f$ and $\ker fg=\ker g$, since $\rk(f)=\rk(g)=2$ and, thus, $fg\in F_{\Gamma,A}$. It follows that $M_{\Gamma,A}$ is a transformation monoid. We define the $n\times (r+1)$ matrix $\widetilde{A}=[A\mid C_{r+1}]$ over $\{0,1\}$ where all entries of the column $C_{r+1}$ are one.
Let the rank of matrix $\widetilde{A}$ be $m_A$. Also, for every $\{\alpha,\beta\}\in E(\Gamma)$ and $1\leq j\leq r$, we define the map $f_{\alpha,\beta,j}\in F_{\Gamma,A}$ as follows:
\begin{equation*}
f_{\alpha,\beta,j}(\gamma)= \begin{cases}
\alpha& \text{if}\ \gamma\in A_{j};\\
\beta& \text{otherwise},
\end{cases}
\end{equation*}
for every $\gamma\in\Omega$.

Let $M\leq \End(\Gamma)$. We say that $M$ is strongly edge transitive if $\{\alpha_1,\alpha_2\},$ $\{\beta_1,\beta_2\}\in E(\Gamma)$ implies there is an element $m\in M$ such that $m\alpha_1=\beta_1$ and $m\alpha_2=\beta_2$. For example, if $\Gamma$ is complete, then $M$ is strongly edge transitive if and only if $M$ is $2$-transitive.

\begin{prop}\label{SEgT}
The monoid $M_{\Gamma,A}$ is strongly edge transitive if and only if $\{\alpha,\beta\}\in E(\Gamma)$ implies the rows $\alpha$ and $\beta$ are not equal in the matrix $A$.
\end{prop}

\begin{proof}
Suppose that $M_{\Gamma,A}$ is strongly edge transitive and $\{\alpha,\beta\}\in E(\Gamma)$. Hence, there is an element $m\in M_{\Gamma,A}$ such that $m\alpha=\beta$ and $m\beta=\alpha$. By construction, we have $m\in F_{\Gamma,A}$. There is an integer $1\leq j\leq r$ such that $\ker m=\{A_j,\overline{A}_j\}$. It follows that $A_{\alpha j}\neq A_{\beta j}$ and, thus, the rows $\alpha$ and $\beta$ are not equal.

Now, suppose that $\{\alpha,\beta\}\in E(\Gamma)$ implies the rows $\alpha$ and $\beta$ are not equal in the matrix $A$, for every $\{\alpha,\beta\}\in E(\Gamma)$. Let $\{\alpha_1,\alpha_2\},\{\beta_1,\beta_2\}\in E(\Gamma)$ be edges of $\Gamma$. Since $\{\alpha_1,\alpha_2\}\in E(\Gamma)$, the rows $\alpha_1$ and $\alpha_2$ are distinct. Hence, there is an integer $1\leq j\leq r$ such that $A_{\alpha_1 j}\neq A_{\alpha_2 j}$. Therefore, we have $(f_{\beta_1,\beta_2,j}(\alpha_1),f_{\beta_1,\beta_2,j}(\alpha_2))=(\beta_1,\beta_2)$ or $(f_{\beta_2,\beta_1,j}(\alpha_1),f_{\beta_2,\beta_1,j}(\alpha_2))=(\beta_1,\beta_2)$.
\end{proof}

We have the following special case when $\Gamma$ is a complete graph.

\begin{cor}\label{2-tran-K}
Suppose that the graph $\Gamma$ is complete. The monoid $M_{\Gamma,A}$ is $2$-transitive if and only if the matrix $A$ has distinct rows.
\end{cor}

\begin{thm}\label{2-tran-Main}
Suppose that if $\{\alpha,\beta\}\in E(\Gamma)$ then the rows $\alpha$ and $\beta$ are not equal in the matrix $A$, for every $\{\alpha,\beta\}\in E(\Gamma)$. The monoid $M_{\Gamma,A}$ satisfies conditions (1),(2),(3) and (5) of Theorem~\ref{aug-simple-monoids} with $\Gamma(M_{\Gamma,A})=\Gamma$ and the augmentation $M_{\Gamma,A}$-submodule $\Aug(\mathbb{F}\Omega)$ is simple if and only if $m_A=\abs{\Omega}$, for every field $\mathbb{F}$.
\end{thm}

\begin{proof}
By Proposition~\ref{SEgT}, the monoid $M_{\Gamma,A}$ is strongly edge transitive. Now, since $C\subset M_{\Gamma,A}$ and the rank of elements $F_{\Gamma,A}$ are $2$, by Proposition~\ref{GraGa}, $F_{\Gamma,A}$ is the minimal $\J$-class of $M_{\Gamma,A}\setminus C$ which is regular and satisfies conditions (3) and (5) of Theorem~\ref{aug-simple-monoids}. Moreover $\Gamma(M_{\Gamma,A})=\Gamma$. Hence, the monoid $M_{\Gamma,A}$ satisfies conditions (1),(2),(3) and (5) of Theorem~\ref{aug-simple-monoids}.

By construction, since $F_{\Gamma,A}$ is the minimal $\J$-class of $M_{\Gamma,A}\setminus C$, the incidence matrix is $[A\mid\overline{A}]$ where $\overline{A}=[\overline{A}_1\cdots \overline{A}_r]$. But $A_j+\overline{A}_j=C_{r+1}$, for all $1\leq j\leq r$, shows $\rank([A\mid\overline{A}])=\rank(\widetilde{A})=m_A$. Now, by Theorem~\ref{aug-simple-monoids}, the augmentation submodule $\Aug(\mathbb{F}\Omega)$ is simple if and only if $m_A=\abs{\Omega}$.
\end{proof}

Similarly, by Corollary~\ref{2-tran-K}, we have the following corollary.

\begin{cor}\label{2-tran-K2}
Suppose that the graph $\Gamma$ is complete. If the matrix $A$ has distinct rows then $M_{\Gamma,A}$ is $2$-transitive and satisfies conditions (1),(2),(3) and (5) of Theorem~\ref{aug-simple-monoids} and the augmentation $M_{\Gamma,A}$-submodule $\Aug(\mathbb{F}\Omega)$ is simple if and only if $m_A=\abs{\Omega}$, for every field $\mathbb{F}$.
\end{cor}

In the following examples $M_{\Gamma,A}$ satisfies conditions (1),(2),(3) and (5) and does not satisfy condition (4) of Theorem~\ref{aug-simple-monoids}.

\begin{thm}
Let $4\leq n$, $\Omega=\{\alpha_1,\ldots,\alpha_n\}$, $r=n-2$, $\Gamma$ be a graph on $\Omega$ and $A$ be the $n\times r$ matrix as follows:\\
\[
\begin{bmatrix}
1 & 1 & 1 & \ldots & 1\\
1 & 0 & 0 & \ldots & 0\\
0 & 1 & 0 & \ldots & 0\\
\cdots\\
0 & 0 & 0 & \ldots & 1\\
0 & 0 & 0 & \ldots & 0
\end{bmatrix}
\]
The monoid $M_{\Gamma,A}$ satisfies conditions (1), (2), (3) and (5) of Theorem~\ref{aug-simple-monoids} with $\Gamma(M)=\Gamma$, but
$M_{\Gamma,A}$ does not satisfy condition (4) of Theorem~\ref{aug-simple-monoids}. In particular, if $\Gamma$ is complete, then $M_{\Gamma,A}$ is $2$-transitive but $\Aug(\mathbb{C}\Omega)$ is not simple.
\end{thm}

\begin{proof}
Since $4\leq n$, all distinct rows $\alpha$ and $\beta$ are not equal in the matrix $A$.
Hence, by Theorem~\ref{2-tran-Main}, $M_{\Gamma,A}$ satisfies conditions (1), (2), (3) and (5) of Theorem~\ref{aug-simple-monoids} for every field $\mathbb{F}$.
The rank of matrix $\widetilde{A}$ is at most $r+1=n-1$. Therefore, by Corollary~\ref{2-tran-K2}, $M_{\Gamma,A}$ does not satisfy condition (4) of Theorem~\ref{aug-simple-monoids}.

Also, if the graph $\Gamma$ is complete, then by Corollary~\ref{2-tran-K}, the monoid $M_{\Gamma,A}$ is $2$-transitive. But, again as the rank of matrix $\widetilde{A}$ is less than $n$, $\Aug(\mathbb{C}\Omega)$ is not simple.
\end{proof}

We now give a natural family of transformation monoids satisfying (1), (2), (3) and (5) of Theorem~\ref{aug-simple-monoids} but nor (4).

Recall that a function $f\colon \mathcal{P}(A)\rightarrow \mathcal{P}(B)$ is a lattice homomorphism, if $f$ satisfies the following conditions:
\begin{enumerate}
\item $f(X\cup Y) = f(X) \cup f(Y)$,
\item $f(X\cap Y) = f(X) \cap f(Y)$,
\end{enumerate}
for all $X, Y\subseteq A$.

\begin{thm}\label{2n-n}
Let $X$ be a set with $2\leq\abs{X}$ and let
$$M_X=\{f\colon \mathcal{P}(X)\rightarrow\mathcal{P}(X)\mid f\mbox{ is a lattice homomorphism}\}.$$
The monoid $M_X$ satisfies conditions (1),(2),(3) and (5) and $M_X$ does not satisfy condition (4) of Theorem~\ref{aug-simple-monoids}, for every field $\mathbb{F}$.
\end{thm}

\begin{proof}
Let $\Gamma$ be a graph with $V(\Gamma)=\mathcal{P}(X)$ and there is an edge between $X_1$ and $X_2$ if and only if $X_1\subsetneq X_2 $ or $X_2\subsetneq X_1$. Suppose that $f\in M_X$. Let $X_1,X_2\in\mathcal{P}(X)$ with $\{X_1,X_2\}\in E(\Gamma)$. Hence, we have $X_1\subsetneq X_2 $ or $X_2\subsetneq X_1$. Since $f$ is a lattice homomorphism, we have $f(X_1\cup X_2)=f(X_1)\cup f(X_2)$. Then, we have $f(X_1)\cup f(X_2)=f(X_1)$ or $f(X_1)\cup f(X_2)=f(X_2)$. It follows that $f(X_1)=f(X_2)$ or $\{X_1,X_2\}\in E(\Gamma)$ and, thus, $f\in\End(\Gamma)$. Therefore, $M_X\leq \End(\Gamma)$.

Since $M_X$ consists of all constant maps on $\mathcal{P}(X)$, $M_X$ satisfies conditions (1) and (3) of Proposition~\ref{GraGa}.

Suppose that $X_1\subsetneq X_2$ and $Y_1\subsetneq Y_2$, for some subsets $X_1,X_2,Y_1,Y_2\in \mathcal{P}(X)$. Since $X_1\subsetneq X_2$, there is an element $x\in X$ such that $x\in X_2\setminus X_1$.
Let $f\colon \mathcal{P}(X)\rightarrow\mathcal{P}(X)$ be a map as follows:
\begin{equation*}
f(Y)= \begin{cases}
Y_2& \text{if}\ x\in Y\\
Y_1& \text{otherwise},
\end{cases}
\end{equation*}
for every $Y\subseteq X$. It is easy to check that $f$ is a lattice homomorphism and, thus, $f\in M_X$.
Now, as $\{f(X_1),f(X_2)\}=\{Y_1,Y_2\}$, $M_X$ satisfies condition (2) of Proposition~\ref{GraGa}. Also, the existence of the map $f$, for some subsets $X_1,X_2,Y_1,Y_2\in \mathcal{P}(X)$, yields $M_X$ satisfies condition (4) of Proposition~\ref{GraGa}.

Now, by Proposition~\ref{GraGa}, $I(M_X)$ is the set of all constant maps of $M_X$ and $M_X\setminus I(M_X)$ has a unique minimal $\J$-class $J$ which is regular, consists of all elements of $M_X$ with rank two and satisfies conditions (3) and (5) of Theorem~\ref{aug-simple-monoids} and moreover $\Gamma(M_X)=\Gamma$.
Suppose that $f\in M_X$ and the rank of $f$ is equal to $2$.
Since $f$ is a lattice homomorphism and $\rk(f)=2$, we have $\im f=\{f(\emptyset),f(X)\}$.
Again, as $f$ is a lattice homomorphism, there is an element $a\in X$ such that $f(\{a\})=f(X)$. Moreover, $a$ is unique, since $f(\{a\})\cap f(\{b\})=f(\{a\}\cap\{b\})=f(\{\emptyset\})$.
Hence, we have $f(Y)=f(X)$, for every subset $Y\subseteq X$ with $a\in Y$ and $f(Z)=f(\emptyset)$, for every subset $Z\subseteq X$ with $a\not\in Z$. Therefore, we have $\ker f=\{A,\mathcal{P}(X)\setminus A\}$ where $A=\{X\in\mathcal{P}(X)\mid a\in X\}$. It follows that the rank of the incidence  matrix of $$\{B\mid B= f^{-1}f\omega, \mbox{ for some }\omega\in\Omega\mbox{ and }f\in E(J)\}$$ is at most $\abs{X}$ over $\mathbb{F}$. Now, as $\abs{V(\Gamma(M_X))}=\abs{\mathcal{P}(X)}=2^n$, $M_X$ does not satisfy condition (4) of Theorem~\ref{aug-simple-monoids}.
\end{proof}

Let $\Omega=\{1,\ldots,n\}$ and $\Delta$ be an acyclic connected digraph with the vertex set $\Omega$. We define a partial order
on $\Omega$ by $\omega\leq\omega'$ if there is a directed path $\omega$ to $\omega'$ in the graph $\Delta$.

Let $B$ be an $n\times r$ matrix over $\{0,1\}$ with no zero columns or all ones columns and
each column is the characteristic vector of an upper set of $(\Omega,\leq)$.

Also, we define the set of mappings
\begin{align*}
\overrightarrow{F}_{\Delta,B}=\{f\colon \Omega\rightarrow\Omega\mid&~ f \mbox{ is a simplicial map, } f \mbox{ preserves orientation and }\\
&~\ker f=\{B_j,\overline{B}_j\},
\mbox{ for some integer } 1\leq j\leq r\}
\end{align*}
and the monoid
$\overrightarrow{M}_{\Delta,B}=\overrightarrow{F}_{\Delta,B}\cup C\cup \{1_{\Omega}\}\in \End(\Delta)$
where $B_j=\{\alpha\in \Omega\mid B_{\alpha j}=1\}$, $\overline{B}_j=\Omega\setminus B_j$ and $C$ is the set of all constant maps on $\Omega$.
Suppose that $f\in \overrightarrow{F}_{\Delta,B}$ and $\ker f=\{B_j,\overline{B}_j\}$, for some integer $1\leq j\leq r$.
Since the digraph $\Delta$ is connected, $f$ is a simplicial map, $f$ preserves orientation and $B_{j}$ is an upper set on $\Omega$, there is an edge form $f(\overline{B}_j)$ to $f(B_{j})$ and there is no edge form $f(B_{j})$ to $f(\overline{B}_j)$.

Let $f,g\in \overrightarrow{F}_{\Delta,B}$. If $fg\not\in C$, then $\im fg=\im f$ and $\ker fg=\ker g$. Since $\rk(f)=\rk(g)=2$, $fg$ is a simplicial map and $fg$ preserves orientation, we have $fg\in \overrightarrow{F}_{\Delta,B}$. It follows that $\overrightarrow{M}_{\Delta,B}$ is a transformation monoid.
Let $\widetilde{B}$ be an $n\times (r+1)$ matrix $[B\mid C_{r+1}]$ over $\{0,1\}$ where all entries of the column $C_{r+1}$ are one and let the rank of matrix $\widetilde{B}$ be $m_B$.

\begin{thm}\label{2-tran-Main-digraph}
Suppose that if $\{\alpha,\beta\}\in E(\Delta)$ then the rows $\alpha$ and $\beta$ are not equal in the matrix $B$, for every $\{\alpha,\beta\}\in E(\Delta)$. The monoid $\overrightarrow{M}_{\Delta,B}$ satisfies conditions (1),(2),(3) and (5) of Theorem~\ref{aug-simple-monoids} and the augmentation $\Aug(\mathbb{F}\Omega)$ is simple if and only if $m_B=\abs{\Omega}$, for every field $\mathbb{F}$.
\end{thm}

\begin{proof}
We prove that $\overrightarrow{M}_{\Delta,B}$ is edge transitive.
Suppose that there are directed edges from $v$ to $v'$ and $w$ to $w'$ in the graph $\Delta$. Since $\{v,v'\}\in E(\Delta)$, the rows $v$ and $v'$ are not equal and, thus, there exists an integer $1\leq j\leq r$ such that $B_{vj_1}\neq B_{v'j_1}$.
Now, as $B_{j}$ is an upper set and there is an edge from $v$ to $v'$, we have $v\in \overline{B}_{j}$ and $v'\in B_j$. Let $f$ be a map with
$f(B_j)=w'$ and $f(\overline{B}_{j})=w$. Again, as $B_{j}$ is an upper set and there is an edge from $w$ to $w'$, the map $f$ is a simplicial map and $f$ preserves orientation. Hence $f\in \overrightarrow{F}_{\Delta,B}$. Now, as $f(v)=w$ and $f(v')=w'$, $\overrightarrow{M}_{\Delta,B}$ is edge transitive.
%
The rest of the proof is entirely similar to the proof of Theorem~\ref{2-tran-Main}.
\end{proof}

\begin{thm}
Suppose that $\abs{\Omega}=2$ or $\abs{\Omega}=3$. If the transformation monoid $(M,\Omega)$ is $2$-transitive then the augmentation submodule $\Aug(\mathbb{C}\Omega)$ is simple.
\end{thm}

\begin{proof}
If $\abs{\Omega}=2$, then $\dim(\Aug(\mathbb{C}\Omega))=1$, so there is nothing to prove.
Hence, we have $\abs{\Omega}=3$. If $M$ is a group, there is nothing to prove. If $M$ is not a group, since $M$ is $2$-transitive, by Proposition~\ref{2-tran}, $M$ satisfies all conditions (1), (2), (3) and (5) of Theorem~\ref{aug-simple-monoids}. Hence,
$I(M)$ consists of all constant maps and $M\setminus I(M)$ has a unique minimal $\J$-class $J$ which is regular.
If $J$ is the group of units $G$, then by (3) $\Aug(\mathbb{C}\Omega)$ is a simple $\mathbb{C}G$-module and, hence, a simple $\mathbb{C}M$-module. So assume that $J$ consists of rank $2$ elements.
Suppose that $e\in E(J)$ and $\ker e=\{\{\alpha,\beta\},\{\gamma\}\}$ where $\Omega=\{\alpha,\beta,\gamma\}$.
Hence, we have $1_{\{\gamma\}}=1_{e^{-1}e\gamma}$.
Since $M$ is $2$-transitive, there exists an element $m\in M$ such that $me\alpha=\alpha$ and $me\gamma=\beta$.
The rank of element $me$ is equal to $2$. Thus, we have $me\in J$.
Now, as $J$ is regular, there is an idempotent $f\in E(J)$ such that $me\R f$ and so $f(\Omega)=\{\alpha,\beta\}$. Thus $\ker f=\{\{\alpha,\gamma\},\{\beta\}\}$ or $\ker f=\{\{\alpha\},\{\beta,\gamma\}\}$.
We suppose that $\ker f=\{\{\alpha,\gamma\},\{\beta\}\}$ as the other case is similar. Hence, we have $1_{\{\beta\}}=1_{f^{-1}f\beta}$ and
$1_{\{\alpha\}}=1_{f^{-1}f\gamma}-1_{e^{-1}e\gamma}$.
Therefore,
$$\mathcal{W}=\langle 1_B\mid B= f^{-1}f\omega, \mbox{ for some }\omega\in\Omega\mbox{ and }f\in E(J)  \rangle_{\mathbb{C}}=\mathbb{C}^{\Omega}.$$
Thus, $M$ satisfies condition (4) of Theorem~\ref{aug-simple-monoids} and, so, the augmentation submodule $\Aug(\mathbb{C}\Omega)$ is simple.
\end{proof}

The examples of this section might create the impression that all examples of transitive transformations monoids with simple augmentation modules, aside from groups, contain rank $2$ mappings.  But this is not the case. Let $\mathbb F$ be a finite field and let $M$ be the monoid of all affine mappings $x\mapsto ax+b$ with $a,b\in \mathbb F$.  Then $M$ consists of the constant mappings and the group of invertible affine mappings.  The latter group is $2$-transitive and hence the augmentation submodule is simple over $\mathbb C$.  But if $\mathbb F$ has more than two elements, then $M$ does not contain any rank $2$ mappings.


\section{$0$-transitive monoids and partial transformation modules}
Let $\Lambda$ be a non-empty set, $PT_{\Lambda}$ be the monoid of all partial transformations of $\Lambda$ and $\Omega=\Lambda\cup\{\omega_0\}$ where $\omega_0\not\in \Lambda$. Let $M$ be a submonoid of $PT_{\Lambda}$. We define the finite transformation monoid $(M,\Omega)$ as follows:
\begin{enumerate}
\item if $m\omega$ is  not defined then $m\omega=\omega_0$, for every $\omega\in\Lambda$ and element $m\in M$;
\item $m\omega_0=\omega_0$, for every element $m\in M$.
\end{enumerate}
In~\cite[Section 7.3]{Ben-Transformation-Monoids-2010}, the partial transformation module $\mathbb{F}\Lambda=\mathbb{F}\Omega/\langle\omega_0\rangle_{\mathbb{F}}$ is defined. Also, it is proved that $\mathbb{F}\Omega/\langle\omega_0\rangle_{\mathbb{F}}$ is isomorphic with $\Aug(\mathbb{F}\Omega)$. Hence, the $M$-module $\mathbb{F}\Lambda$ is isomorphic with $\Aug(\mathbb{F}\Omega)$. Hence, we have the following theorem.

\begin{thm}\label{partial-simple}
The module $\mathbb{F}\Lambda$ is simple if and only if $\Aug(\mathbb{F}\Omega)$ is simple, for a field $\mathbb{F}$.
\end{thm}

Let $(M,\Omega)$ be an $0$-transitive monoid. By~\cite[Theorem 4.1]{Ben-Transformation-Monoids-2010}, $M$ has a zero map and $I(M)=\{0\}$ and by~\cite[Theorem 4.6]{Ben-Transformation-Monoids-2010}, $M$ has a
unique minimal non-zero $\J$-class $J$ which is regular and $J\cup\{0\}$ acts $0$-transitively (as a semigroup) on $\Omega$.

\begin{thm}\label{0-tran-Gamma(M)}
Let $(M,\Omega)$ be an $0$-transitive monoid with $M\omega_0 = \{\omega_0\}$ and let $J$ be the unique minimal non-zero $\J$-class of $M$.
The following items hold:
\begin{enumerate}
\item If the rank of $J$ is equal to $2$ then the graph $\Gamma(M)$ is a star graph with the star vertex $\omega_0$.
\item The augmentation submodule $\Aug(\mathbb{F}\Omega)$ is simple, for a field $\mathbb{F}$, if and only if the rank of $J$ is equal to $2$ and $M$ satisfies condition (4) of Theorem~\ref{aug-simple-monoids}.
\item If the augmentation submodule $\Aug(\mathbb{F}\Omega)$ is simple, for a field $\mathbb{F}$, then the maximal subgroup of $J$ is trivial.
\end{enumerate}

\end{thm}

\begin{proof}
(1) Let $\alpha\in\Omega\setminus\{\omega_0\}$. Since $J$ is $0$-transitive, there exists an element $m\in J$ such that $m\alpha=\alpha$. There exists an idempotent $e\in E(J)$ such that $e\R m$. It follows that $e\alpha=\alpha$. Now, as $e\omega_0=\omega_0$, we have $\{\alpha,\omega_0\}\in E(\Gamma(M))$. Now, suppose that $\{\alpha,\beta\}\in E(\Gamma(M))$. Then, there is an idempotent $e\in E(J)$ such that $e\alpha=\alpha$ and $e\beta=\beta$. Since the rank of $e$ is equal to $2$ and $e\omega_0=\omega_0$, we have $\alpha=\omega_0$ or $\beta=\omega_0$. It follows that the graph $\Gamma(M)$ is the star graph with the star vertex $\omega_0$.

(2) Suppose that the augmentation submodule $\Aug(\mathbb{F}\Omega)$ is simple.
Let $e\in E(J)$. Since the augmentation submodule $\Aug(\mathbb{F}\Omega)$ is simple, by Theorem~\ref{aug-simple-monoids}, $\Aug(\mathbb{F}e\Omega)$ is a simple $\mathbb{F}G_e$-module.
Now, by~\cite[Section 7.3]{Ben-Transformation-Monoids-2010}, we have $\Aug(\mathbb{F}e\Omega)\cong \mathbb{F}[e\Omega\setminus \{\omega_0\}]$.
Hence, $\mathbb{F}[e\Omega\setminus \{\omega_0\}]$ is simple as an $\mathbb{F}G_e$-module.
As $\mathbb{F}[e\Omega\setminus \{\omega_0\}]$ is simple, we have $\abs{e\Omega\setminus \{\omega_0\}}=1$ and, thus, $\abs{e\Omega}=2$. It follows that the rank of $J$ is equal to $2$.

Now, suppose that the rank of $J$ is equal to $2$ and $M$ satisfies condition (4) of Theorem~\ref{aug-simple-monoids}. By part (1), the graph $\Gamma(M)$ is connected. Hence, by Theorem~\ref{aug-simple-monoids}, the augmentation submodule $\Aug(\mathbb{F}\Omega)$ is simple.

(3) Let $e\in E(J)$. By part (2), the rank of $J$ is equal to $2$. Since $G_e$ acts faithfully on $e\Omega$ and fixes $\omega_0$, from $\abs{e\Omega}=2$, we deduce $G_e$ is trivial.
\end{proof}

In the following example, there exists a $0$-transitive monoid $(M,\Omega)$ which the augmentation submodule $\Aug(\mathbb{F}\Omega)$ of $M$ is simple, for every field $\mathbb{F}$, and $M$ has an element with rank more than $2$.

\begin{example}\label{0-tran-aug}
Let $2\leq n$, $\Omega=\{\omega_0,\omega_1,\ldots,\omega_n\}$ and $$M=\langle \{m_1,m_k,m_{ij}\mid 1\leq k\leq n, 1\leq i,j\leq n\mbox{ and } i\neq j\}\rangle$$ where $m_1,m_k,m_{ij}$ are defined as follows:
\begin{align*}
m_1\omega_0&=\omega_0,m_1\omega_l=\omega_l,\mbox{ for all }1\leq l\leq n,\\
m_k\omega_0&=\omega_0,m_k\omega_k=\omega_0,m_k\omega_l=\omega_l,\mbox{ for all }1\leq l\leq n\mbox{ with }l\neq k\mbox{ and}\\
m_{ij}\omega_0&=\omega_0,m_{ij}\omega_i=\omega_j,m_{ij}\omega_l=\omega_l,\mbox{ for all }1\leq l\leq n\mbox{ with }l\neq i\mbox{ and }l\neq j.
\end{align*}
The augmentation submodule $\Aug(\mathbb{F}\Omega)$ of the $0$-transitive monoid $(M,\Omega)$ is simple, for every field $\mathbb{F}$.
\end{example}

\begin{proof}
Let $W$ be a submodule of $\Aug(\mathbb{F}\Omega)$ and $\sum c_p\omega_p\in W$ with $\sum c_p\omega_p\neq 0$.
Since $\sum c_p\omega_p\neq 0$, there exists an integer $1\leq i\leq n$ such that $c_i\neq 0$. Let $d=\sum c_p-c_i$.
We have
$$d\omega_0+c_i\omega_i=m_n\cdots m_{i+1}m_{i-1}\cdots m_1 \sum c_p\omega_p.$$
Since $\sum c_p\omega_p\in\Aug(\mathbb{F}\Omega)$, we have $\sum c_p=0$, and, thus, $d=-c_i$. It follows that $\omega_i-\omega_0\in W$. As  $\omega_j-\omega_0=m_{ij}(\omega_i-\omega_0)$, we have $\omega_j-\omega_0\in W$, for every $1\leq j\leq n$ with $j\neq i$. Now, as $\omega_{i_1}-\omega_{i_2}=(\omega_{i_1}-\omega_0)-(\omega_{i_2}-\omega_0)$, for every distinct elements $1\leq \omega_{i_1},\omega_{i_2}\leq n$, we have $\Aug(\mathbb{F}\Omega)=W$.
\end{proof}

\begin{thm}\label{meet-omega-0}
Let $(\Omega,\wedge)$ be a finite meet semilattice with minimum $\omega_0$ and let $M$ be its endomorphism monoid.
The augmentation submodule $\Aug(\mathbb{F}\Omega)$ is simple, for every field $\mathbb{F}$.
\end{thm}

\begin{proof}
We have
$$M=\{f\colon \Omega\rightarrow\Omega \mid f(\alpha\wedge\beta)=f(\alpha)\wedge f(\beta)\mbox{ and }f(\omega_0)=\omega_0\}.$$
Let $\alpha, \beta\in\Omega\setminus\{\omega_0\}$.
We define the map $f_{\alpha,\beta}\colon \Omega\rightarrow\Omega$, for every $x\in\Omega$, as follows:\\
\begin{equation*}
f_{\alpha,\beta}(x)= \begin{cases}
\beta& \text{if}\ \alpha\leq x;\\
\omega_0& \text{otherwise}.
\end{cases}
\end{equation*}
Since $f_{\alpha,\beta}\in M$ and $f_{\alpha,\beta}(\alpha)=\beta$, $M$ is $0$-transitive with unique minimal non-zero $\J$-class $J$ of rank $2$.

We define the map $$\phi\colon \Omega\rightarrow \{ B\mid B= f^{-1}f(\omega), \mbox{ for some }\omega\in\Omega\mbox{ and }f\in E(J)\},$$ for every $\alpha\in\Omega$, as follows:
\begin{equation*}
\phi(\alpha)= \begin{cases}
f_{\alpha,\alpha}^{-1}f_{\alpha,\alpha}(\alpha)& \text{if}\ \alpha\neq\omega_0;\\
f_{\beta,\beta}^{-1}f_{\beta,\beta}(\omega_0)& \text{otherwise},
\end{cases}
\end{equation*}
where $\beta\neq \omega_0$.
As $\phi$ satisfies conditions of Proposition~\ref{GraGa2}, $M$ satisfies condition (4) of Theorem~\ref{aug-simple-monoids}. Therefore, by Theorem~\ref{0-tran-Gamma(M)}, the augmentation submodule $\Aug(\mathbb{F}\Omega)$ is simple.
\end{proof}

The symmetric inverse monoid $I_{\Lambda}$ on a set $\Lambda$ is the monoid of all partial injective mappings of $\Lambda$.  It is also known as the rook monoid~\cite{Solomonrook}.  Using the previous theorem, we can recover the following well-known result.

\begin{thm}\label{Inv-Lam}
Let $I_{\Lambda}$ be the monoid of all partial injective maps on $\Lambda$ and $\mathbb{F}$ be a field.
The $\mathbb{F}I_{\Lambda}$-module $\mathbb{F}\Lambda$ is simple.
\end{thm}
\begin{proof}
Let $\Omega=\Lambda\cup\{\omega_0\}$.
We define the meet semilattice $(\Omega,\wedge)$ as follows:
$$\alpha\leq\beta\mbox { if and only if }\alpha=\omega_0,$$
for every $\alpha,\beta\in\Omega$. The element $\omega_0$ is the minimum.  Viewing $I_{\Lambda}$ as a submonoid of $T_{\Omega}$, we show that $I_{\Lambda}=\End(\Omega)$.
Suppose that there exists $f\in \End(\Omega)$ with $f\not\in I_{\Lambda}$. Then there exist elements
$\alpha\neq\beta\in \Lambda$ such that $f(\alpha)=f(\beta)\neq\omega_0$. We then have $f(\alpha)=f(\alpha)\wedge f(\beta)=f(\alpha\wedge\beta)=f(\omega_0)=\omega_0$, a contradiction. Conversely, if $f\in I_{\Lambda}$, then $f(\alpha\wedge \beta)=f(\alpha)\wedge f(\beta)$ for all $\alpha,\beta\in \Omega$ because either $\alpha=\beta$ or both sides are $\omega_0$.  Thus $I_{\Lambda}= \End(\Omega)$. Now, by Theorem~\ref{meet-omega-0}, the augmentation $\Aug(\mathbb{F}\Omega)$ is simple. Hence, by Theorem~\ref{partial-simple}, $\mathbb{F}\Lambda$ is simple.
\end{proof}

\begin{cor}
Let $PT_{\Lambda}$ be the monoid of all partial transformations on $\Lambda$ and $\mathbb{F}$ be a field.
Then the module $\mathbb F\Lambda$ is simple.
\end{cor}

\begin{proof}
Since $I_{\Lambda}\subseteq PT_{\Lambda}$, the result follows from Theorems~\ref{partial-simple} and~\ref{Inv-Lam}.
\end{proof}

We now use Rees matrix semigroups to show that the incidence matrix in condition (4) of Theorem~\ref{aug-simple-monoids} can be essentially arbitrarily complicated in the $0$-transitive case.  We use here Rees matrix semigroups.

\begin{thm}\label{Rees-Simple}
Let $\Lambda=\{1,\ldots, n\}$ and let $\Omega=\Lambda\cup\{\omega_0\}$ with $\omega_0\not\in\Lambda$.
Let $A$ be an $n\times r$ matrix over $\{0,1\}$ with no zero columns and no equal columns and let $M=M'\cup\{1_{\Lambda}\}$ where $M'$ is the Rees matrix semigroup $\mathcal{M}^{0}(\{1\},n,r;A^{T})$.
Define the transformation monoid $(M,\Omega)$ as follows:
\begin{align*}
(i,1,j)\alpha&=\begin{cases}
i& \text{if}\ \alpha\neq\omega_0\ \text{and}\ A_{\alpha j}\neq 0;\\
\omega_0& \text{otherwise},
\end{cases}\\
1_{\Lambda}\alpha&=\alpha,\\
0\alpha&=\omega_0,
\end{align*}
for every $\alpha\in\Omega$ and $(i,1,j)\in M'$.
Then the augmentation submodule $\Aug(\mathbb{F}\Omega)$ is simple if and only if $\rank(A)=n$, for a field $\mathbb{F}$. Moreover, if $\rank(A)=n$ then $M$ is $0$-transitive.
\end{thm}

\begin{proof}
The fact that $A$ has no repeated columns or zero columns implies that $M$ acts faithfully on $\Omega$, as is readily checked.
First, suppose that $\rank(A)=n$. Hence, the matrix $A$ has no zero rows. Now, as the matrix $A$ has no zero columns,
the Rees matrix semigroup $M'$ is regular and so $M'\setminus \{0\}$ is a $\J$-class of $M$. Note that $(i,1,j)$ is an idempotent if and only if $A_{ij}\neq 0$. As there are no zero columns, for each $j\in \Lambda$, there is some idempotent of the form $(i,1,j)$.

Since $$\ker (i,1,j)=\{\{\alpha\mid A_{\alpha j}=0\}\cup\{\omega_0\},\{\beta\mid A_{\beta j}\neq 0\}\},$$ for every $(i,1,j)\in M'$, the matrix $A$ has no zero columns and $M'\setminus \{0\}$ is a $\J$-class of $M$, it follows that $M'\setminus \{0\}$ is the unique $0$-minimal ideal of $M$ and is of rank $2$.  We verify that $M$ is $0$-transitive.
Let $\alpha,\beta\in\Omega\setminus\{\omega_0\}$. Since the matrix $A$ has no zero columns, there exists an integer $1\leq j\leq r$ such that $A_{\alpha j}\neq 0$. Hence, $(\beta,1,j)\alpha=\beta$. Now, as $M\omega_0=\omega_0$, the monoid $M$ is $0$-transitive. Let $(i,1,j)\in E(J)$.
Since $$\ker (i,1,j)=\{\{\alpha\mid A_{\alpha j}=0\}\cup\{\omega_0\},\{\beta\mid A_{\beta j}\neq 0\}\},$$
and $(i,1,j)\omega_0=0$, the incidence matrix of the set system $$\{B\mid B= f^{-1}f\omega, \mbox{ for some }\omega\in\Omega\mbox{ and }f\in E(J)\}$$ is the matrix $\left[
\begin{array}{c|c}
    A & \overline{A} \\ \hline
    0,\ldots,0 & 1,\ldots,1
\end{array}
\right]$
where $\overline{A}=[\overline{A}_1\cdots \overline{A}_r]$. Since $\rank(A)=n$, the rank of the incidence matrix is equal to $n+1$. Therefore, by Theorem~\ref{0-tran-Gamma(M)}.(2), the augmentation submodule $\Aug(\mathbb{F}\Omega)$ is simple.

Now, suppose that $\Aug(\mathbb{F}\Omega)$ is simple. Note that $I(M)=\{0\}$.  Hence, by Theorem~\ref{aug-simple-monoids}, the monoid $M$ contains a unique minimal non-zero $\J$-class $J$ and moreover $J$ is regular. Obviously, $J\subseteq M'\setminus \{0\}$. Since $J$ is regular, there exists an idempotent
 $(i,1,j)\in E(J)$. Let $(i',1,j')\in M'$. Since the matrix $A$ has no zero columns, there exists an integer $1\leq \alpha\leq n$ such that $A_{\alpha j}\neq 0$. Also, since $(i,1,j)\in E(J)$, we have $A_{ij}\neq 0$. Hence, we have
\begin{eqnarray}\label{i-j}
(i',1,j)(i,1,j)(\alpha,1,j')=(i',1,j').
\end{eqnarray}
Since $I(M)=\{0\}$, it follows by minimality of $J$ that $(i',1,j')\in J$ and so $J=M'\setminus \{0\}$.
Let $(i,1,j)\in E(J)$.
Since $$\ker (i,1,j)=\{\{\alpha\mid A_{\alpha j}=0\}\cup\{\omega_0\},\{\beta\mid A_{\beta j}\neq 0\}\}$$ and $(i,1,j)\omega_0=0$,
the incidence matrix of the set system $$\{B\mid B= f^{-1}f\omega, \mbox{ for some }\omega\in\Omega\mbox{ and }f\in E(J)\}$$ is the matrix $\left[
\begin{array}{c|c}
    A & \overline{A} \\ \hline
    0,\ldots,0 & 1,\ldots,1
\end{array}
\right]$. By Theorem~\ref{aug-simple-monoids}, the rank of the incidence matrix is equal to $\abs{\Omega}=n+1$ over $\mathbb{F}$. Hence $\rank(A)=n$.
\end{proof}

\begin{rmk}
Theorem~\ref{Rees-Simple} can also be deduced from classical semigroup representation theory as per~\cite{RhodesZalc}.  Assuming that $A$ has no zero rows, one has that $M'\setminus \{0\}$ is a unique minimal $\mathscr J$-class with trivial maximal subgroup.  It therefore is the apex of a unique simple module.  The construction of that module in~\cite{RhodesZalc} is as a quotient of $\mathbb F\Lambda$ and the quotient is proper if and only if the matrix $A$ does not have rank $n$.  Since $\mathbb F\Lambda\cong \Aug(\mathbb F\Omega)$, this gives the desired conclusion.
\end{rmk}

\section*{Acknowledgments}
The first author was supported, in part, by CMUP (UID/MAT/00144/ 2013), which is funded by FCT (Portugal) with national (MCTES) and European structural funds through the programs FEDER, under the partnership agreement PT2020 and also partly supported by the FCT post-doctoral scholarship (SFRH/BPD/89812/2012) and the second author was supported by NSA MSP \#H98230-16-1-0047 and PSC-CUNY. This work was performed while the first author was visiting the City College of New York. He thanks the College for its warm hospitality.

\def\malce{\mathbin{\hbox{$\bigcirc$\rlap{\kern-7.75pt\raise0,50pt\hbox{${\tt
  m}$}}}}}\def\cprime{$'$} \def\cprime{$'$} \def\cprime{$'$} \def\cprime{$'$}
  \def\cprime{$'$} \def\cprime{$'$} \def\cprime{$'$} \def\cprime{$'$}
  \def\cprime{$'$} \def\cprime{$'$}


\end{document}